\documentclass{article}

\usepackage[T1]{fontenc}     
\usepackage[utf8]{inputenc}
\usepackage[a4paper]{geometry}
\usepackage{amsmath}
\usepackage{mathtools}
\usepackage{graphicx}
\usepackage[english]{babel}
\usepackage{hyperref}

\usepackage{amssymb,amsfonts,amsthm}
\usepackage{dsfont}
\usepackage{xcolor}
\usepackage{appendix}

\mathtoolsset{showonlyrefs}

\newcommand{\N}{\mathbb N}
\newcommand{\Z}{\mathbb Z}
\newcommand{\R}{\mathbb R}

\newtheorem{lemma}{Lemma}
\newtheorem{proposition}{Proposition}
\newtheorem{theorem}{Theorem}
\newtheorem*{theorem*}{Theorem}
\newtheorem{remark}{Remark}
\newtheorem{definition}{Definition}
\newtheorem{corollary}{Corollary}

\numberwithin{equation}{section}

\def\bu{{\overline u}}

\def \bv{{\overline v}}

\def \bphi{{ \overline \phi}}

\def \be{{\overline e}}
\def \tk{{\tilde k}}

\def \D{{\mathcal D}}
\def \bD{{\overline{\D}}}

\def\cD{\mathcal{D}}
\DeclareMathOperator{\dive}{div}
\def\cC{\mathcal{C}}
\def \cE{{\mathcal E}}

\definecolor{fgreen}{RGB}{34,139,34}

\title{Boundary stabilization of one-dimensional cross-diffusion systems in a moving domain: linearized system}
\usepackage{authblk}
\author[1,2,*]{Jean Cauvin-Vila}
\author[1,2]{Virginie Ehrlacher}
\author[1]{Amaury Hayat}
\affil[1]{\footnotesize{CERMICS, Ecole des Ponts ParisTech, 6 \& 8 avenue Blaise Pascal, 77455~Marne-la-Vall\'ee, France}}
\affil[2]{\footnotesize{MATHERIALS Team-project, Inria Paris, 2 Rue Simone Iff, 75012 Paris}}
\affil[*]{\emph{Corresponding author}: jean.cauvin-vila@enpc.fr}
\date{\empty}

\providecommand{\keywords}[1]
{
  \small	
  \textbf{\textit{Keywords---}} #1
}

\begin{document}

\maketitle

\begin{abstract}
    We study the boundary stabilization of one-dimensional cross-diffusion systems in a moving domain. We show first exponential stabilization and then finite-time stabilization in arbitrary small-time of the linearized system around uniform equilibria, provided the system has an entropic structure with a symmetric mobility matrix. One example of such systems are the equations describing a Physical Vapor Deposition (PVD) process. This stabilization is achieved with respect to both the volume fractions and the thickness of the domain. The feedback control is derived using the backstepping technique, adapted to the context of a time-dependent domain. In particular, the norm of the backward backstepping transform is carefully estimated with respect to time.
\end{abstract}

\medskip

\section{Introduction}\label{sec:intro}

Cross-diffusion systems naturally arise in diffusion models of multi-species mixtures in a wide variety of applications: tumor growth, population dynamics, materials science etc., see for example Chapter 4 of~\cite{jungel2016a} for an introduction to these systems. Let $n\geq 1$ so that the number of species in the system of interest is $n+1$,
$d\in \mathbb{N}^*$ the spatial dimension and $\Omega \subset \mathbb{R}^d$ the bounded spatial domain occupied by the mixture. Such a cross-diffusion system then models the evolution of $u_i(t,x)$ for all $0\leq i \leq n$, where $u_i(t,x)$ denotes the local concentration or volume fraction of the $i^{th}$ species in the mixture at a time $t>0$ and point $x\in \Omega$. Setting $\widetilde{u}:=(u_0,\cdots,u_n)^T$, a typical cross-diffusion system reads as follows (together with appropriate initial and boundary conditions):
\begin{equation}\label{eq:crossdiff}
    \partial_t \widetilde{u} - \dive_x \left( \widetilde{A}\left(\widetilde{u}\right) \nabla_x \widetilde{u} \right) = 0 \quad \mbox{ for }(t,x)\in\mathbb{R}_+^*\times \Omega ,
\end{equation}
for some matrix-valued application $\widetilde{A}: \mathbb{R}^{n+1} \to \mathbb{R}^{(n+1)\times (n+1)}$. Significant advances in the understanding of the mathematical structure of these systems have been achieved in the last ten years. Indeed, it has been understood in the seminal works~\cite{burger2010,jungel2012,jungel2013,jungel2015} that many of these systems have an \emph{entropic structure}, which enables to obtain appropriate estimates in order to prove the existence of weak solutions to systems of the form \eqref{eq:crossdiff}.

These systems arise in particular in materials science, in order to model atomic diffusion within solids. Indeed, hydrodynamic limits of some stochastic lattice hopping models~\cite{quastel1992} read as cross-diffusion systems of the form \eqref{eq:crossdiff}. Our work here is mainly based on the study initiated in~\cite{bakhta2018}, where the authors considered a one-dimensional cross-diffusion system defined in a moving boundary domain in order to model a Physical Vapor Deposition process (PVD) used for the fabrication of thin film layers in the photovoltaic industry. The process can be described as follows: a wafer is introduced in a hot chamber where chemical elements are injected under gaseous form. As the latter deposit on the substrate, a heterogeneous solid layer grows upon it. Because of the high temperature conditions, diffusion occurs in the bulk until the wafer is taken out and the system is frozen.

In this model, the solid layer is composed of $n+1$ different chemical species and occupies a domain of the form $(0,e(t)) \subset \R_+$, where $e(t) >0$ denotes the thickness of the film. For all $0\leq i \leq n$, let $\phi_i \in L^1_{\rm loc}(\mathbb{R}_+^*)$ be a non-negative function so that $\phi_i(t)$ represents the flux of atoms of species $i$ absorbed at the surface of the film layer at time $t$. The evolution of the thickness of the film is determined by the fluxes $(\phi_i)_{0\leq i\leq n}$ and reads as:

\begin{equation}
    \label{domain}
    e(t) = e_0 + \int_0^t \sum_{i=0}^n \phi_i(s)ds,
\end{equation}
where $e_0>0$ denotes the initial thickness of the film. The local volume fractions of the different species $u_0(t,x),\dots,u_n(t,x)$ are naturally expected to satisfy the following constraints:
\begin{equation}
    \label{constraints}
    \forall~ 0\leq i \leq n, ~ u_i(t,x)\geq0 ~ \text{and} ~ \sum_{j=0}^n u_j(t,x) = 1.
\end{equation}
These constraints \eqref{constraints} allows one to equivalently express $u_0$ as $1-\sum_{i=1}^n u_i$. As a consequence, the whole system can be equivalently rewritten using the unknown vector $u:=(u_1,\dots,u_n)^T$. More precisely, denoting by $\varphi$ the vector-valued function $(\phi_1,\dots,\phi_n)^T$, the cross-diffusion system in the solid layer reads:
\begin{equation}
    \label{cross-diff}
    \left\{
    \begin{aligned}
        \partial_t u - \partial_x (A(u)\partial_x u) & = 0,          & \text{for} ~ (t,x) \in \bigcup_{t\in \mathbb{R}_+^*} \{t\} \times (0,e(t)) ~ =: U_{e}, \\
        (A(u)\partial_x u)(t,0)                      & = 0,          & \text{for} ~ t \in \R_+^*,                                                  \\
        (A(u)\partial_x u)(t,e(t)) + e'(t)u(t,e(t))  & = \varphi(t), & \text{for} ~ t \in \R_+^*,                                                  \\
        u(0,x)                                       & = u^0(x),     & x \in (0,e_0),
    \end{aligned}
    \right.
\end{equation}
for some matrix-valued application $A: \mathbb{R}^n \to \mathbb{R}^{n\times n}$ which is called the \emph{diffusion matrix} of the system, and some initial condition $u^0\in L^\infty((0,e_0), \mathcal D)$ where the set of constraints $\mathcal D$ is defined below in \eqref{eq:defD}. The boundary conditions express that the system is isolated at $x=0$ but that there is an incoming (vector-valued) flux $\varphi(t)$ at $x=e(t)$ where the extra term $e'(t)u(t,e(t))$ accounts for the growth of the layer.

\medskip

The main motivation of the authors of~\cite{bakhta2018} for the study of such a system was ultimately to be able to control the gaseous fluxes $(\phi_0,\dots,\phi_n)$ injected during the PVD process in order to reach target composition profiles. The global existence of weak solutions to system \eqref{cross-diff} was shown by adapting the \emph{boundedness-by-entropy method}~\cite{jungel2015}. The authors also proved existence of solutions to an optimization problem related to the control of the fluxes and long-time asymptotics of the volume fraction profiles in the case of constant external fluxes (i.e. when the system is in open-loop). However, it is not clear whether the corresponding asymptotic profiles are exponentially stable in open-loop, and whether one could use a better flux control $\varphi$ to improve their stability remains an open question. The main difficulty lies in the expansion of the domain with time and the coupling between $u(t)$ and $e(t)$. When the domain is fixed, results concerning the exponential convergence to equilibrium of solutions to \eqref{eq:crossdiff} were already proven for several diffusion matrices $A$ (see \cite{chen2006,jungel2012,zamponi2017,alasio2020}) and in particular recently for the PVD cross-diffusion matrix \cite{hopf2022multi}.
\medskip

This work is concerned with the stabilization of the linearized version of \eqref{cross-diff} around uniform equilibrium states (precisely defined in Sections \ref{sec:properties}, \ref{sec:linearized}, \ref{sec:definitions}), under appropriate assumptions on the cross-diffusion matrix $A$. These assumptions build on the usual entropic structure conditions stated in~\cite{jungel2015,bakhta2018} to address the nonlinear problem, together with an additional symmetry assumption on the \emph{mobility matrix} of the system at the state considered. In particular, the PVD model in~\cite{bakhta2018} satisfies these conditions around any such state. In this paper, we show that we can obtain finite-time stabilization of the linearized system, with a feedback control derived using the backstepping technique inspired from \cite{coron2017}.

\medskip

First introduced in \cite{backstepping1,backstepping2,backstepping3} for finite dimensional systems, the backstepping approach was later used and adapted for PDE in \cite{backsteppingadapt,BaloghKrstic,BBK2003,krstic2004,KrsticBook}. It consists in transforming the original system, hard to stabilize, into a simpler target system, using an isomorphism. The main difficulty is then to show the existence of such an isomorphism. The usual backstepping approach for PDE, presented in \cite{KrsticBook}, searches for isomorphisms under the form of a Volterra transform of the second kind (see \eqref{eq:defT}), which are conveniently always invertible, among other advantages. Some attempt to introduce a generalized backstepping approach which does not necessarily rely on Volterra transforms have also been introduced in \cite{CoronLu,GagnonCoronMorancey,GagnonMarx,Zhang1,Zhang2,CHXZ2021,GHXZ2021}.
The Volterra approach has been used in many areas and for many systems in the last decades including parabolic equations (see for instance \cite{baccoli2015,coron2017,espitia2019boundary}), hyperbolic system (see for instance \cite{Krstichyperbolic,2011-Vazquez-Coron-Krstic-Bastin-CDC,JF2016,auriol2018delay,Krstic2016,Hu2015backstepping,CHS2021}), etc.
However, no result exists on diffusion system of the form \eqref{cross-diff} where the domain extends with time (in a way that is not compensated in the dynamics). The reason is that this situation brings new difficulties, in particular the backstepping transform has to depend on time and one has to make sure that this does not jeopardize the exponential stability (resp. the finite time stability). Indeed, when applying the transformation backward to obtain the exponential stability of the original system, the cost of the estimate depends on the norm of the backstepping transform and of the norm of its inverse, which depends itself on time (see \eqref{stability_loop}). If this norm goes to infinity exponentially fast, it could be that the original system is not exponentially stable, let alone finite-time stable, even though the target system is. One can still note \cite{izadi2015,izadi2015b} where the authors also consider a backstepping approach applied to a parabolic equation in a time-dependent domain. However, their situation is different thanks to their dynamics, and in both cases they do not consider the norm of the inverse of the backstepping transform. Concerning parabolic equations with time-dependent coefficients, one can also note the work by Smyshlyaev and Krstic \cite{smyshlyaev2005control} which considers a heat equation with a time-dependent reactivity and the work by Kerschbaum and Deutscher \cite{kerschbaum2019backstepping} where the authors consider the exponential stability of a system of parabolic equations with a diagonal diffusion and a time-dependent reactivity. In both cases the difficulty lies in the existence of a time-dependent kernel for the transform and is dealt by converting the kernel equations into an integral equation, using fixed point and successive approximations. We do not use such a method here as we aim for a more generic method that could be applied to more complicated systems and steady-states, and that can allow estimates such as \eqref{eq:estimate-k} that are so crucial to reach finite time stabilization.

\medskip

The outline of the paper is the following: we first recall the main mathematical properties of the moving boundary cross-diffusion system introduced in~\cite{bakhta2018} in Section~\ref{sec:preliminaries} and present the linearized version of this system we focus on in this work. Our main theoretical results are gathered in Section~\ref{sec:mainres}. The description of the backstepping transformation we consider here is given in Section~\ref{sec:backstepping}. Proofs of our results are gathered in Section~\ref{sec:proofs}. Additional details and some technical calculations are added in the Appendices.

\section{Preliminaries}\label{sec:preliminaries}

The aim of this section is to recall the main mathematical properties of the system studied in~\cite{bakhta2018} and to introduce the linearized version of this system we focus on in this work. In Section~\ref{sec:entropy}, we recall the assumptions needed on the diffusion matrix $A$ for the associated cross-diffusion system to have an entropic structure and state the additional assumptions required by the stabilization analysis presented in this work. Mathematical properties of system \eqref{domain}-\eqref{cross-diff} are discussed in Section~\ref{sec:properties}. Finally, the linearized version of system~\eqref{domain}-\eqref{cross-diff} which we will focus on in this article is introduced in Section~\ref{sec:linearized}.

\medskip

\noindent \textbf{Notations:}
Let us first introduce some useful notation. For any continuous non-decreasing positive function $\widetilde{e}: \mathbb{R}_+ \to \mathbb{R}_+^*$ and any $T>0$, we define the sets $U_{\widetilde{e}}:= \cup_{t \in \R_+^*} \{t\}\times (0,\widetilde{e}(t))$ and $U_{\widetilde{e}}^T:= \cup_{t \in (0,T)} \{t\}\times (0,\widetilde{e}(t))$, the time-space moving domains we consider in this paper. For any $0 < T \leq +\infty$, any $1 \leq p,q \leq \infty$, any $k \in \Z$, we denote by  $\left[L^p((0,T),W^{k,q}) \right]_{\tilde{e}}$ the set of measurable functions $f$ from $U_{\tilde{e}}^T$ to $\R$ such that respectively: if $p < \infty$
\[ \left(\int_0^T \|f(t)\|_{W^{k,q}(0,\tilde{e}(t))}^p dt \right)^{1/p} < \infty, \]
and if $p = \infty$,
\[ \sup_{0 \leq t \leq T} \|f(t)\|_{W^{k,q}(0,\tilde{e}(t))} < \infty.  \]
These quantities define norms, naturally denoted by  $\|\cdot\|_{\left[L^p((0,T),W^{k,q}) \right]_{\tilde{e}}}$, which in turn induce a Banach structure. We generalize this notation for functions defined in more general time intervals $(t_1,t_2)$ for $0\leq t_1 \leq t_2 \leq T$. The sets $\left[L^p_{\rm loc}((0,T),W^{k,q}) \right]_{\tilde{e}}$ are defined similarly. The space $\left[\mathcal C^0((0,T),L^p) \right]_{\tilde{e}}$ is defined as the set of functions $f: U_{\widetilde{e}}^T \to \mathbb{R}$ such that the function $(0,T) \times (0,1) \ni (t,x) \mapsto f(t,x\widetilde{e}(t))$ belongs to $\mathcal C^0((0,T); L^p(0,1))$.

\subsection{Entropic structure of the nonlinear system}\label{sec:entropy}

We detail in this section the assumptions needed on the diffusion matrix $A$ to get existence of a weak solution to the nonlinear cross-diffusion system \eqref{domain}-\eqref{cross-diff} and introduce some additional notations. These assumptions coincide with the requirements highlighted in~\cite{jungel2015,bakhta2018} for system \eqref{domain})-\eqref{cross-diff} to have an entropic structure. We refer to \cite{jungel2015,desvillettes2015} for more details about the entropic structure of cross-diffusion systems, and to \cite{chen2019} for a discussion about necessary and sufficient conditions for a cross-diffusion system to admit such a structure.

Let $\cD\subset \R^n$ be defined by
\begin{equation}\label{eq:defD}
    \cD := \left \{  (u_1, \cdots,u_n) \in (\R^*_+)^n, \quad \sum \limits_{i=1}^n u_i <1\right \} \subset (0,1)^n.
\end{equation}
Note that a solution $u$ to \eqref{domain}-\eqref{cross-diff} satisfies the constraints \eqref{constraints} if and only if $u(t,x) \in \bD$, for all $t\in \mathbb{R}_+^*$ and all $x\in (0, e(t))$. Note also that, in view of \eqref{constraints}, the strict inequalities in $\eqref{eq:defD}$ imply that the $n+1$ volume fractions are positive. The following set of assumptions on the diffusion matrix $A$ allows guaranteeing that the corresponding cross-diffusion system enjoys a favorable entropic structure.

\medskip

\bfseries Assumptions: \normalfont

\begin{itemize}
    \item [(H0)] $A \in \cC^0(\overline{\cD},\R^{n\times n})$;
    \item [(H1)] there exists a bounded from below strictly convex function $h \in \cC^0(\overline{\cD})$ such that $h\in \cC^2(\cD, \R)$, such that its derivative $Dh:\cD \to \R^n$ is invertible in $\R^n$ and such that (the symmetric part of) the matrix $D^2h(u) A(u)$ is positive semi-definite for all $u \in \cD$;
    \item[(H2)] moreover, there exists $\alpha >0$, and for all $1\leq i \leq n$, there exists $1\geq m_i>0,$ such that for all $z = (z_1, \cdots, z_n)^T \in \R^n$ and $u=(u_1,\cdots, u_n)^T\in \cD$,
          $$
              z^T D^2h(u)A(u) z \geq \alpha \sum_{i=1}^n u_i^{2m_i-2} z_i^2.
          $$
\end{itemize}

The interested reader may consult \cite{jungel2015,bakhta2018,burger2010}, let us briefly comment on these assumptions here. A function $h$ such that (H1) and (H2) hold is called an \emph{entropy density} of the cross-diffusion system. The associated \emph{entropy functional} $\mathcal{E}$ is then defined by
\begin{equation}
    \label{eq:defE}
    \cE : \left\{
    \begin{array}{ccc}
        L^\infty(\Omega; \overline{\cD}) & \longrightarrow & \R                                     \\
        u                                & \longmapsto     & \cE (u): =  \int_{\Omega} h(u(x))\,dx, \\
    \end{array}
    \right .
\end{equation}
and for all $u\in  L^\infty(\Omega; \overline{\cD})$, we identify the differential $D\mathcal{E}(u)$ with its Euclidean gradient which is equal to the function $Dh(u)$.

\medskip

The first equation of system \eqref{cross-diff} can then be formally rewritten under the following form:

\begin{equation} \label{eq:GF}
    \partial_t u - \dive_x \left( M(u) \nabla_x D\cE(u) \right) = 0, \quad \mbox{ for } (t,x) \in U_e= \bigcup_{t\in \mathbb{R}_+^*} \{t\} \times (0,e(t)),
\end{equation}
where $M:\cD \to \R^{n\times n}$ is the so-called \textit{mobility matrix} of the system and is defined for all $u\in \cD$ by
\begin{equation} \label{eq:defM}
    M(u):= A(u) (D^2h(u))^{-1}.
\end{equation}

From formulation \eqref{eq:GF} and under assumption (H1), one can check that $\cE$ is a Lyapunov functional of the system, which justifies the term ``entropy functional''. The fact that $Dh$ is invertible allows one to work with the so-called entropic variables $w := Dh(u)$ and to automatically get a solution $u \in \cD$ compatible with the constraints. Finally, under assumption (H2), (the symmetric part of) the mobility matrix $M(u)$ is even positive definite, so that the formulation \eqref{eq:GF} is even coercive and one can derive gradient estimates. In physical applications, this entropic structure has a thermodynamic interpretation and in particular the entropic variables $w$ are strongly linked to the notion of chemical potential (see Appendix~A in \cite{jungel2015}).

\begin{remark} \label{rq:pvd}
    One particular example of diffusion matrix $A$ is studied in~\cite{bakhta2018} for the PVD model used in photovoltaics applications. This diffusion matrix is defined as follows: for all $u:=(u_i)_{1\leq i \leq n}\in \mathbb{R}^n$, $A(u) = \left( A_{ij}(u)\right)_{1\leq i,j \leq n} \in \mathbb{R}^{n\times n}$ where
    \begin{equation}\label{eq:defA}
        \left\{
        \begin{aligned}
             & \forall 1\leq i \leq n, ~ A_{ii}(u) = \sum_{1 \leq j \neq i \leq n} (K_{ij}-K_{i0})u_j + K_{i0}, \\
             & \forall 1\leq i \neq j \leq n, ~ A_{ij}(u) = -(K_{ij}-K_{i0})u_i.
        \end{aligned}
        \right.
    \end{equation}
    where, for all $0 \leq i \neq j \leq n$, the positive real numbers $K_{ij}$ satisfy $K_{ij} = K_{ji} > 0$ and represent the cross-diffusion coefficients of atoms of type $i$ with atoms of type $j$. Note that $A(u)$ is \emph{not} a symmetric matrix in general.
    It is proved in~\cite{bakhta2018} that the diffusion matrix defined by \eqref{eq:defA} satisfies assumptions (H0)-(H1)-(H2), with $m_i=\frac{1}{2}$ for all $1\leq i \leq n$ and with the function $h$ being defined as the classical Boltzmann entropy density:
    \begin{equation}\label{eq:defh}
        h : \left\{
        \begin{array}{ccc}
            \overline{\cD}            & \longrightarrow & \R                                                                      \\
            u:=(u_i)_{1\leq i \leq n} & \longmapsto     & h(u) =   \sum \limits_{i=1}^n u_i \log u_i + (1-\rho_u)\log (1-\rho_u), \\
        \end{array}
        \right .
    \end{equation}
    where for all $u=(u_i)_{1\leq i \leq n} \in \mathbb{R}^n$, $\rho_u := \sum_{i=1}^{n} u_i$. Furthermore, the mobility matrix associated to \eqref{eq:defA} and \eqref{eq:defh} is given for $u \in \D$ as $M(u)=(M_{ij}(u))_{1 \leq i,j \leq n}$, where:
    \begin{equation*}
        \left\{
        \begin{aligned}
             & \forall 1 \leq i \leq n, ~ M_{ii}(u) = \sum_{1 \leq j \neq i \leq n} K_{ij} u_i u_j + K_{i0}u_i(1-\rho_u), \\
             & \forall 1 \leq i \neq j \leq n, ~ M_{ij}(u) = -K_{ij}u_i u_j.
        \end{aligned}
        \right.
    \end{equation*}
    Note that $M(u)$ is always a symmetric matrix.
\end{remark}

\subsection{Main mathematical properties of the nonlinear model}\label{sec:properties}

The aim of this section is to recall the main mathematical properties of the nonlinear model \eqref{domain}-\eqref{cross-diff} and highlight the open questions on the control and stabilization of this system that are of interest here.

\medskip

It was proved in~\cite{bakhta2018} that there exists at least one weak solution to system \eqref{domain}-\eqref{cross-diff} satisfying the constraints \eqref{constraints} in the following sense:
\begin{theorem}[Theorem 2 of \cite{bakhta2018}]\label{th:moving}
    Assume $A$ satisfies assumptions (H0)-(H1)-(H2) and let $h: \bD \to \mathbb{R}$ be the associated function so that (H1) and (H2) are satisfied. Let us assume that $u^0\in L^\infty((0,e_0); \cD)$ satisfies $w^0:= (Dh)(u^0) \in L^\infty((0,e_0); \R^n)$.
    Let us also assume that $(\phi_0, \cdots , \phi_n)\in L^{\infty}_{\rm loc}(\R^*_+; \R_+^{n+1})$.
    Then, there exists a weak solution $u$ with initial condition $u^0$ to \eqref{domain}-\eqref{cross-diff}
    such that for almost all $(t,x)\in U_{e}$, $u(t,x)\in \overline{\cD}$, and
    \[
        u\in \left[L^2_{\rm loc}(\R^*_+; H^1)^n\right]_e \quad \mbox{ and }\quad \partial_t u \in \left[L^2_{\rm loc}(\R^*_+; (H^1)')^n\right]_e.
    \]
\end{theorem}

In the case when the fluxes $(\phi_i)_{0\leq i \leq n}$ are constant in time, it is legitimate to wonder if the volume fraction profiles $(u_i)_{0\leq i\leq n}$ will converge to some constant profiles, and if yes, at which rate. The following result was proved in~\cite{bakhta2018} under the assumption that the entropy density $h$ of the system is given by \eqref{eq:defh}.

\begin{proposition}[Proposition 1 of \cite{bakhta2018}]\label{prop:longtime}
    Let us assume that the assumptions of Theorem~\ref{th:moving} hold together with the following ones:
    \begin{itemize}
        \item[(T1)] for all $0\leq i \leq n$, there exists $\bphi_i > 0$ so that $\phi_i(t) = \bphi_i$, for all $t\in \R_+$;
        \item[(T2)] the entropy density $h$ can be chosen so that for all $u\in \overline{\mathcal D}$,  $h(u)$ is defined by \eqref{eq:defh}.
    \end{itemize}
    Let us define
    \begin{equation}
        \label{eq:speed}
        \bv := \sum_{i=0}^n \bphi_i,
    \end{equation}
    and for all  $0\leq i \leq n$, 
    \begin{equation}
        \label{eq:stationary-fraction}
        \bu_i:= \frac{\bphi_i}{\bv}
    \end{equation}
    so that $\bu:=(\bu_i)_{1\leq i \leq n}\in \cD$. Then, there exists a constant $C>0$ such that for almost all $t\geq 0$,
    \[\forall 1\leq i \leq n, \quad \frac{1}{\be(t)}\|u_i(t,\cdot) - \bu_i\|_{L^1(0,\be(t))} \leq \frac{C}{\sqrt{t+1}}, \]
    and
    \[ \frac{1}{\be(t)}\left\|\left( 1 - \rho_{u(t,\cdot)}\right) - \bu_0\right\|_{L^1(0,\be(t))} \leq \frac{C}{\sqrt{t+1}},\]
    where $\overline{e}(t)$ is the thickness at time $t$ of the layer, given by
    \begin{equation}
        \label{eq:stationary-domain}
        \overline{e}(t) = \be_0 + \int_0^t \sum_{i=0}^n \phi_i(s)\,ds = \be_0 + t \bv,
    \end{equation}
    for some value of the initial thickness of the layer $\be_0>0$.
\end{proposition}

Let us make a few comments about this result.
\begin{itemize}
    \item In this specific case, the thickness of the boundary layer $\be(t)$ grows linearly with constant speed $\bv$.
    \item Proposition~\ref{prop:longtime} does not state that the quantity $\|u_i(t,\cdot) - \bu_i\|_{L^1(0,\be(t))}$ goes to $0$ as $t$ goes to infinity, it only enables to guarantee the existence of a constant $C>0$ such that
          \[ \forall t>0, \quad \|u_i(t,\cdot) - \bu_i\|_{L^1(0,\be(t))} \leq C\sqrt{t}. \]
          Proposition~\ref{prop:longtime} still states that the \emph{rescaled} volume fraction profiles converge to constants in the long-time limit. More precisely, denoting by $v(t,y) := u(t,\be(t)y)$ for all $t>0$ and $y\in (0,1)$ and by $v_i$ the $i^{th}$ component of $v$ for $1\leq i \leq n$, it holds that
          \begin{equation}
              \label{eq:average}
              \forall 1\leq i \leq n, \quad \|v_i(t,\cdot) - \bu_i\|_{L^1(0,1)} = \frac{1}{\be(t)} \|u_i(t,\cdot) - \bu_i\|_{L^1(0,\be(t))} \leq \frac{C}{\sqrt{t+1}}.
          \end{equation}
    \item In the case of constant fluxes $(\bphi_i)_{0\leq i \leq n}$, i.e. when the thickness of the film at all time $t>0$ is equal to $\overline{e}(t)$, and when the initial condition $u^0$ is equal to $\bu$, it can be easily checked that the function $u$ defined by $u(t,x) = \bu$ for all $t>0$ and $x\in (0, \overline{e}(t))$ is a solution to system \eqref{domain}-\eqref{cross-diff}. Therefore, we use the denomination ``target state of \eqref{cross-diff}'' to refer to a couple $(\bu,\mathbb{R}_+ \ni t \mapsto \be(t))$.
\end{itemize}

The preceding remarks provided the main source of motivation for this work about the stabilization of system \eqref{domain}-\eqref{cross-diff}. Assuming that the initial condition $u^0$ at time $t=0$ is chosen as a small perturbation of $\bu$, of the form $u^0 = \bu + \delta u^0$, and that the initial thickness of the film $e_0$ at time $t=0$ is a small perturbation of the initial thickness $\overline{e}_0$, i.e. $e_0 = \be_0 + \delta e_0$, does there exist a set of feedback fluxes $(\phi_i)_{0\leq i \leq n}$ such that for a time $t$ large enough, the volume fraction profiles $u(t)$ and thickness of the system $e(t)$ converge to $\bu$ and $\overline{e}(t)$ in a stronger norm than the average $L^1$ norm used in \eqref{eq:average} ? In other words, can the system be stabilized around the target state $(\bu, \overline{e})$ and at which rate ? Can exponential stability or finite-time stability be achieved, i.e. can the system be stabilized at a rate much better than the one provided by the strategy which would consist in keeping the fluxes $\phi_i$ constant and equal to $\overline{\phi_i}$ as considered in Proposition~\ref{prop:longtime} ?

\medskip

This work can be seen as an important first step in this direction. Indeed, we provide answers on the stabilization of a \itshape linearized \normalfont version of the system \eqref{domain}-\eqref{cross-diff}. From this result, we expect to be able to obtain the \emph{local} stabilization of the original nonlinear system in a future work.

\subsection{Linearized system and control variables}\label{sec:linearized}

The aim of this section is to introduce the linearized system which is the main focus of this paper, together with an appropriate change of control variables that enables to decouple the control analysis of the volume fractions and the thickness of the domain.
\medskip

Recall that we consider small perturbations $(\delta u^0,\delta e_0)$ at $t=0$ around the initial condition $\bu$ given by \eqref{eq:stationary-fraction} and initial thickness $\overline{e}_0$. Assuming that the imposed fluxes on the system are of the form $\phi_i(t) = \bphi_i + \delta \phi_i(t)$ for all $0\leq i \leq n$ and $t>0$, we wish to investigate the linearized dynamic of $(\delta u(t,\cdot), \delta e(t))$ which can be seen as first-order corrections of $(u(t,\cdot)-\bu,e(t)-\be(t))$, where $\be$ is given by \eqref{eq:stationary-domain}. Recall also the notation \eqref{eq:speed} for the growth speed of the layer $\be$.

\medskip

Then, the first order correction of the thickness reads, for all $t \geq 0$:
\begin{equation}
    \label{thickness}
    \delta e(t) =  \int_0^t \sum_{i=0}^{n} \delta \phi_i(s) ds + \delta e_0, \quad \mbox{ and } \quad \delta e'(t) = \sum_{i=0}^{n} \delta \phi_i(t).
\end{equation}

In addition, the first-order corrections of the system \eqref{cross-diff} around the target state $(\bu, \overline{e})$ yields the following system, the solution of which is $\delta u$, for given $\delta u^0$, $\delta \varphi:= (\delta \phi_1, \cdots, \delta \phi_n)^T$ and $\delta \phi_0$:
\begin{equation}
    \label{linearized}
    \left\{
    \begin{aligned}
        \partial_t \delta u - A(\bu) \partial_{xx}^2 \delta u       & = 0,             & \text{for} ~ (t,x) \in U_{\be},          \\
        A(\bu)\partial_x \delta u(t,\be(t)) + \bv\delta u(t,\be(t)) & = \delta\psi(t), & \text{for} ~ t \in \R_+^*,               \\
        A(\bu) \partial_x \delta u(t,0)                             & = 0,             & \text{for} ~ t \in \R_+^*,               \\
        \delta u(0,x)                                               & = \delta u^0(x), & ~ \text{for} ~ x \in (0,\overline{e}_0).
    \end{aligned}
    \right.
\end{equation}
where for any $t \geq 0$,
\begin{equation}
    \label{effective_control}
    \delta \psi(t) := \delta \varphi(t) - \delta e'(t)  \bu \in \R^n.
\end{equation}

Remark that the solution $\delta u$ to system \eqref{linearized} only depends on the $n$ independent control variables denoted by $\delta \psi = (\delta \psi_i)_{1 \leq i \leq n}$. Therefore, since we originally had $n+1$ control variables $(\delta \phi_{i})_{i\in\{0,..,n\}}$, it remains an extra degree of freedom. This degree of freedom ought to be designed exclusively for the stabilization of the thickness $\delta e$. We make this explicit by defining a new control variable as for any $t \geq 0$:
\begin{equation}
    \label{eq:thickness-variable}
    \delta \theta (t) := \sum_{i=0}^n \delta \phi_i(t),
\end{equation}
such that for any $t\geq0$:
\begin{equation} \label{eq:thickness}
    \delta e(t) = \int_0^t \delta \theta(s)ds + \delta e_0
\end{equation}
Now we claim that the change of control variables $(\delta \phi_0,\dots,\delta \phi_n) \to (\delta \theta, \delta \psi_1,\dots,\delta \psi_n)$, defined according to \eqref{effective_control} and \eqref{eq:thickness-variable} is invertible. Indeed, it can be checked that for any $t > 0$,
\[ \delta \varphi(t) = \delta \psi(t) + \delta \theta(t) \bu, \]
and
\[\delta \phi_0(t) = \delta \theta(t) - \sum_{i=1}^n \left( \delta \psi_i(t) + \delta \theta(t) \bu_i \right).  \]

This new choice of control variables $(\delta \theta, \delta \psi)$ is more convenient for our analysis since we can now completely decouple the control analysis of the thickness and of the volume fractions respectively, as will be made clear in Section~\ref{sec:proofs}.

\section{Stabilization of the linearized system: main results}\label{sec:mainres}

The aim of this section is to present the main results of this work, which focuses on the stabilization of the linearized system \eqref{thickness}-\eqref{linearized}. In Section~\ref{sec:definitions}, we introduce the precise notions of weak solutions and stability considered here. In Section~\ref{sec:results} are stated our main theoretical results, and we decompose the problem into $n$ \emph{scalar} problems.
Finally, in Section~\ref{sec:backstepping}, we detail our backstepping strategy to stabilize the scalar problem.

\subsection{Main definitions}\label{sec:definitions}

We first need to specify the notion of solution to system \eqref{linearized} we will consider here. In the following, we are interested in the stabilization with the spatial $L^2$ norm, so defining an appropriate notion of weak solution in $L^2$ for $L^2$ initial data is needed for our analysis to hold. In our case, anticipating slightly on the next section, the fluxes will be defined as a closed-loop feedback of the form
\begin{equation}\label{eq:flux-eq}
    \delta\psi(t) = \Psi(t,\delta u(t)),
\end{equation}
where $\delta u(t) = \delta u(t,\cdot)$ is the solution function at time $t$ defined in the space domain $(0,\be(t))$ and where the application $\Psi$ is decomposed into a non-local integral part and a local multiplication operator at $x = \overline{e}(t)$ (recall the expression \eqref{eq:stationary-domain} of $\be(t)$). More precisely, the application $\Psi$ will be of the following form: for almost all $t \geq 0$ and all $z\in H^1(0,\overline{e}(t))^n$,
\begin{equation}\label{abstract-feedback}
    \Psi(t,z):= H_{nl}(t)z + H_l(t)z,
\end{equation}
where the family of operators $(H_{nl}(t))_{t \geq 0}$ and $(H_l(t))_{t \geq 0}$ will be required to satisfy the following properties (in fact, these conditions are necessary to give a meaning to our definition of weak solution, see Definition~\ref{def:weak-solution} below):

\medskip

\bfseries Properties of operators:\normalfont

\begin{itemize}\label{properties-operator}
    \item [(P1)] for almost all $t \geq 0$, $H_{nl}(t)$ is a continuous linear application from $L^2(0,\overline{e}(t))^n$ to $\mathbb{R}^n$;
    \item [(P2)] for all $T>0$, and all $z \in \left[L^2((0,T), (L^2))^n\right]_{\be}$, the application $(0,T)\ni t \mapsto H_{nl}(t)z(t)$ belongs to $L^2(0,T)^n$. Moreover, there exists a constant $C(T)>0$ such that
          $$
              \left\|H_{nl}(\cdot) z(\cdot)\right\|_{\left[L^2((0,T), (L^2))^n\right]_{\be}} \leq C(T) \| z\|_{ \left[L^2((0,T), (L^2))^n\right]_{\be}};
          $$

    \item [(P3)] for almost all $t \geq 0$, the operator $H_l(t): H^1(0, \overline{e}(t))^n \to \mathbb{R}^n$ is defined as follows:
          \begin{equation}\label{eq:Hl}
              \forall z\in  H^1(0, \overline{e}(t))^n,  \quad H_l(t)z:= K_l(t)z(\overline{e}(t))
          \end{equation}
          where $K_l \in L_{loc}^\infty\left(\mathbb{R}_+^*; \mathbb{R}^{n\times n}\right)$ is a locally bounded matrix-valued application.
\end{itemize}

\medskip

Using the particular form of fluxes highlighted above, a weak solution can be defined by testing \eqref{linearized} against regular test functions that satisfy dual boundary conditions (see Definition~\ref{def:weak-solution} below and Appendix~\ref{app:weak} for details). We obtain the following definition:

\begin{definition}[Weak solution in $L^2$]
    \label{def:weak-solution}
    Let $\delta u^0\in L^2(0, \overline{e}_0)$. Let $(H_{nl}(t))_{t \geq 0}$ and $(H_l(t))_{t \geq 0}$ be two families of operators satisfying (P1)-(P2)-(P3). A function $\delta u \in \left[\mathcal C^0(\mathbb{R}_+,L^{2})^n\right]_{\overline{e}}$ such that $\partial_t \delta u \in \left[L^2(\R_+;(H^1)')\right]_{\be}$ is said to be a $L^2$-weak solution to \eqref{linearized} with fluxes $\delta \psi$ defined by \eqref{eq:flux-eq}-\eqref{abstract-feedback} if, for any $T > 0$, it satisfies:

    \[
        \begin{split}
            \int_{0}^{T}\int_{0}^{\be(t)} \delta u(t,x) \cdot \left[\partial_{t}v(t,x)+A(\bu)^T\partial_{xx}^{2}v(t,x) \right] dxdt \\
            + \int_{0}^{\be_0}\delta u^0(x) \cdot v(0,x) dx
            + \int_{0}^{T}(H_{nl}(t)\delta u(t)) \cdot v(t,\be(t))dt = 0,
        \end{split}
    \]

    for any test function $v$ that satisfies:
    \begin{itemize}
        \label{def-Dstar}
        \item $v \in \left[\left(L^2\left((0,T); H^2\right)\right)^n\right]_{\overline{e}}\cap \left[\mathcal C^0([0,T],L^{2})^n\right]_{\overline{e}}$,
        \item $\partial_t v\in \left[\left(L^2\left((0,T); L^2\right)\right)^n\right]_{\overline{e}}$,
        \item $v(T,\cdot) = 0$,
        \item $A(\bu)^T\partial_{x}v(t,0) = 0, ~ \mbox{for almost all }t \in (0,T)$,
        \item $K_l(t)^Tv(t,\overline{e}(t)) - A(\bu)^T\partial_{x}v(t,\be(t)) = 0, \mbox{for almost all } t \in (0,T)$.
    \end{itemize}

    \begin{remark}
        \label{rq:time-derivative}
        One may wonder why the assumption on the time derivative is needed. In fact, we will use this assumption to ensure uniqueness in this class of solutions (see the proof of Corollary~\ref{cor:exp-stability} based on the backstepping transformation). Nevertheless, it is likely that any $L^2$ solution to \eqref{linearized} satisfies this assumption. This would amount to prove a regularity result for \eqref{linearized} (or equivalently a uniqueness result in the class of $L^2$ solutions) that we do not provide in this work. (see however Lemma~\ref{lem:uniqueness} in Appendix~\ref{app:target} about the homogeneous problem)
    \end{remark}
\end{definition}

Similarly, the control of the thickness $\delta \theta$ will also be defined as a closed-loop feedback of the form
\begin{equation}\label{eq:deltatheta}
    \delta \theta(t) = \Theta(t, \delta e(t))
\end{equation}
where the application $\Theta$ will be chosen so that $\Theta \in L^1_{\rm loc}\left( \mathbb{R}_+^* ;\cC^0 (\mathbb{R})\right)$.

\medskip

Let us now give precise definitions for the different notions of stabilization we consider in the present work. We start with the notion of exponential stabilization:

\begin{definition}[Exponential stabilization in $L^2$]
    \label{def:exp-stab}
    Let $\mu>0$. A target state $(\bu,\be)$ of \eqref{cross-diff} is said to be \emph{$\mu$-exponentially stabilizable} in $L^2$ if there exist constants $C_{\bar{u},\mu},C_{\bar{e},\mu} > 0$ such that:
    \begin{itemize}
        \item[a)] There exist families of operators $(H_{nl}(t))_{t\geq0}$ and $(H_l(t))_{t\geq0}$ satisfying properties (P1)-(P2)-(P3), such that, for any perturbation $\delta u^0 \in L^2(0,\overline{e}_0)$, the linearized system \eqref{linearized} with the fluxes defined by \eqref{eq:flux-eq}-\eqref{abstract-feedback} has a unique $L^2$ weak solution $\delta u$ in the sense of Definition~\ref{def:weak-solution} and this solution satisfies:
              \begin{equation}
                  \label{eq:def-stability-concentrations}
                  \|\delta u(t)\|_{L^2(0,\be(t))} \leq C_{\bar{u},\mu} e^{-\mu t} \|\delta u^0\|_{L^2(0,e_0)}, ~ \text{for all} ~ t \geq 0.
              \end{equation}
        \item[b)] There exists a function  $\Theta \in L^1_{\rm loc}\left(\mathbb{R}_+^*;\mathcal C^0(\mathbb{R})\right)$ such that, for any perturbation $\delta e_0 \in \R$, $\delta e$ is well-defined by \eqref{eq:thickness} with $\delta \theta$ defined by \eqref{eq:deltatheta} and satisfies:
              \begin{equation}
                  \label{eq:def-stability-thickness}
                  |\delta e(t)| \leq C_{\bar{e},\mu} e^{-\mu t}|\delta e_0|, ~ \text{for all} ~ t \geq 0.
              \end{equation}
    \end{itemize}
\end{definition}

Let us also give a definition of \emph{finite time} stabilization:

\begin{definition}[Finite time stabilization in $L^2$]
    \label{def:finite-stab}
    Let $T>0$. A target state $(\bu,\be)$ of \eqref{cross-diff} is said to be \emph{stabilizable in finite time $T$} in $L^2$ if:
    \begin{itemize}
        \item[a)] There exist families of operators $(H_{nl}(t))_{t \geq 0}$ and $(H_l(t))_{t \geq 0}$ satisfying properties (P1)-(P2)-(P3), such that, for any perturbation $\delta u^0 \in L^2(0,\be_0),$ the linearized system \eqref{linearized} with the fluxes defined by \eqref{eq:flux-eq}-\eqref{abstract-feedback} has a unique $L^2$ weak solution $\delta u$ in the sense of Definition~\ref{def:weak-solution} and this solution satisfies:
              \begin{itemize}
                  \item[i)] (stability) For any $\epsilon > 0$, there exists $\nu_u > 0$ such that if $\|\delta u^0 \|_{L^2(0,\be_0)} \leq \nu_u$ then for all $t \geq 0$:
                        \begin{equation}
                            \label{eq:def-finite-stability-concentrations}
                            \|\delta u(t) \|_{L^2(0,\be(t))} \leq \epsilon.
                        \end{equation}
                  \item[ii)] (convergence)
                        \begin{equation}
                            \label{T-stability-u}
                            \|\delta u(t)\|_{L^2(0,\be(t))} \to 0 ~ \text{as} ~ t \to T^-.
                        \end{equation}
              \end{itemize}

        \item[b)] There exists a function  $\Theta \in L^1_{\rm loc}\left((0,T) ;\mathcal C^0(\mathbb{R})\right)$ such that, for any perturbation $\delta e_0 \in \mathbb{R}$,  $\delta e$ is well-defined by \eqref{eq:thickness} with $\delta \theta$ defined by \eqref{eq:deltatheta} and satisfies:
              \begin{itemize}
                  \item[i)] (stability) For any $\epsilon > 0$, there exists $\nu_e > 0$ such that if $|\delta e_0| \leq \nu_e$ then for all $t \geq 0$:
                        \begin{equation}
                            \label{def-finite-stability-thickness}
                            |\delta e(t)| \leq \epsilon.
                        \end{equation}
                  \item[ii)] (convergence)
                        \begin{equation}
                            \label{T-stability-e}
                            \delta e(t) \to 0 ~ \text{as} ~ t \to T^-.
                        \end{equation}
              \end{itemize}
    \end{itemize}
\end{definition}

\subsection{Main results}\label{sec:results}
Let us summarize our assumptions here. Let $(\bu,\be)$ be a target state of \eqref{cross-diff} (in the sense of the discussion in Section~\ref{sec:properties}) such that:

\medskip

\textbf{Assumptions:}
\begin{itemize}
    \item[(A1)] $\overline{u}\in \mathcal D$ (which implies that for all $1 \leq i \leq n, ~ \bu_i > 0$ and $1- \rho_{\overline{u}} = 1 - \sum_{i=1}^n \bu_i >0$);
    \item[(A2)] The diffusion matrix application $A$ satisfies assumptions (H0)-(H1)-(H2). Besides, the mobility matrix application $M$ defined by \eqref{eq:defM} is such that $M(\overline{u})$ is symmetric.
\end{itemize}

Let us emphasize here that, in particular, the diffusion matrix $A$ defined by \eqref{eq:defA} in Remark~\ref{rq:pvd} satisfies assumption (A2). The additional requirement that $M(\overline{u})$ is symmetric enables to guarantee that the matrix $A(\overline{u})$ is diagonalizable with positive eigenvalues:

\begin{lemma}
    \label{lem:diagonal}
    Assume that $\bu$ satisfies (A1) and that the diffusion matrix $A$ satisfies (A2). Then it holds that $A(\overline{u})$ is diagonalizable with positive eigenvalues.
\end{lemma}
\begin{proof}
    From \eqref{eq:defM}, it holds that $A(\bu)=M(\bu)H(\bu)$ with $H(\bu):=D^2h(\bu)$. The matrices $M(\bu)$ and $H(\bu)$ are both symmetric positive definite, which implies that $H(\bu)^{1/2}$ is well-defined and
    \[ A(\bu) = M(\bu)H(\bu) = H(\bu)^{-1/2}H(\bu)^{1/2}M(\bu)H(\bu)^{1/2} H(\bu)^{1/2}.\]
    Therefore $A(\bu)$ is similar to the symmetric real matrix $H(\bu)^{1/2}M(\bu)H(\bu)^{1/2}$ that is clearly positive definite. Hence the result.
\end{proof}

\medskip

The result of Lemma~1 enables us to decompose system \eqref{linearized} into $n$ scalar problems as follows. One can write $A(\bu) = Q^{-1}(\bu) \Sigma(\bu) Q(\bu)$, where the coefficients of the diagonal matrix $\Sigma(\bu)$ are the positive eigenvalues $(\sigma_1,\dots,\sigma_n)$ of $A(\bu)$. As a consequence, denoting by $z:=Q(\bu)\delta u$, by $z_i$ the $i^{th}$ component of $z$ for $1\leq i \leq n$ and by $z^0 := Q(\bu)\delta u^0$, system \eqref{linearized} boils down to the following set of $n$ uncoupled scalar equations: for all $1\leq i \leq n$,
\begin{equation}
    \left\{
    \label{linearized_decomp}
    \begin{aligned}
        \partial_t z_i - \sigma_i \partial_{xx}^2 z_i        & = 0,                  & \text{for} ~ (t,x) \in U_{\be}, \\
        \sigma_i\partial_x z_i(t,\be(t)) + \bv z_i(t,\be(t)) & = \delta \psi^{i}(t), & \text{for} ~ t \in \R_+^*,      \\
        \sigma_i(\partial_x z_i)(t,0)                        & = 0,                  & \text{for} ~ t \in \R_+^*,      \\
        z_i(0,x)                                             & = z^0_i(x),           & \text{for} ~ x \in (0,\be_0),
    \end{aligned}
    \right.
\end{equation}
where we have introduced the following change of coordinates of the feedback: for all $t\geq 0$, $\delta \psi^{i}(t) := (Q(\bu) \delta \psi(t))_i$. We are now in position to state our main results.

\begin{theorem}
    \label{thm:1}
    Let $\mu>0$. Let $(\bu,\be)$ be a target state and assume that assumptions (A1)-(A2) are satisfied. Then, $(\bu,\be)$ is $\mu$-exponentially stabilizable in $L^2$ in the sense of Definition~\ref{def:exp-stab}. More precisely, let us introduce the following functions and operators:
    \begin{itemize}
        \item for any $t \geq 0$ and $w\in \mathbb{R}$, $\Theta(t,w) = -\mu w$;
        \item for any $t \geq 0$, $1\leq i \leq n$, $z \in H^1(0,\be(t))$ and $\lambda>0$,
              \begin{align*}
                  H_{l,\lambda}^{i}(t)z & := \sigma_i k^{\sigma_i}_\lambda(\be(t),\be(t)) z(\be(t)), \\ H_{nl, \lambda}^{i}(t)z &:= \int_0^{\be(t)} \left[\sigma_i \partial_x k^{\sigma_i}_\lambda(\be(t),y) + \bv k^{\sigma_i}_\lambda(\be(t),y)\right] z(y) dy,
              \end{align*}
              where $k_\lambda^{\sigma_i}$ is the unique solution to \eqref{eq:kernel_eq} given below with $\sigma = \sigma_i$. We also define for all $t\geq 0$, $\lambda >0$ and $z:=(z_i)_{1\leq i \leq n}\in H^1(0, \overline{e}(t))^n$,
              \begin{align*}
                  H_{l,\lambda}(t)z   & := Q(\overline{u})^{-1} \left( H^{i}_{l,\lambda}(t)z_i\right)_{1\leq i \leq n},  \\
                  H_{nl, \lambda}(t)z & := Q(\overline{u})^{-1} \left( H^{i}_{nl,\lambda}(t)z_i\right)_{1\leq i \leq n}.
              \end{align*}

              Then, there exists $\lambda>0$ large enough such that $\left( \Theta , (H_{l,\lambda}(t))_{t \geq 0}, (H_{nl,\lambda}(t))_{t\geq 0}\right)$ satisfy the conditions of Definition~\ref{def:exp-stab}. %The boundary control law in system \eqref{linearized} is given by $\delta \psi(t) = H_{nl}(t) \delta u(t) + H_l(t) \delta u(t)$.
    \end{itemize}
\end{theorem}

Elaborating on this result, we can even obtain finite time stabilization.

\begin{theorem}
    \label{thm:2}
    Let $(\bu,\be)$ be a target state and assume that assumptions (A1)-(A2) are satisfied. Then, it is stabilizable in any finite time $T>0$ in $L^2$ in the sense of Definition~\ref{def:finite-stab}.
\end{theorem}

\medskip

In Appendix \ref{sec:lyapunov} we show why the common approach which consists in directly using a basic quadratic Lyapunov function would fail to show the exponential stability in this case. This motivates our use of the backstepping approach, described in Section~\ref{sec:backstepping}.

\section{Backstepping approach}\label{sec:backstepping}

The proof of Theorem~\ref{thm:1} and Theorem~\ref{thm:2} relies on the use of a backstepping transformation, in conjunction with the fact that system~\eqref{linearized} can be decomposed into $n$ scalar uncoupled problems of the form \eqref{linearized_decomp}. Thus, we will need to collect intermediate results on the resulting scalar equations, which is the object of the present section.

From now on, let $\tau_1\geq 0$ and let us denote by $U_{\overline{e},\tau_1}:= \bigcup_{t\in (\tau_1, +\infty)} \{t\} \times (0, \be(\tau_1))$, where $\be$ is defined in \eqref{eq:stationary-domain}. Note that $U_{\overline{e},0} = U_{\overline{e}}$. Let us now fix $\sigma, \lambda>0$ and consider the following auxiliary scalar problem:
\begin{equation}
    \left\{
    \label{linearized_scal}
    \begin{aligned}
        \partial_t \zeta_\lambda^\sigma - \sigma \partial_{xx}^2 \zeta_\lambda^\sigma         & = 0,                               & \text{for} ~ (t,x) \in U_{\overline{e},\tau_1}, \\
        \sigma \partial_x \zeta_\lambda^\sigma(t,\be(t)) + \bv \zeta_\lambda^\sigma(t,\be(t)) & = \delta \psi_\lambda^{\sigma}(t), & \text{for} ~ t \in (\tau_1, +\infty),           \\
        \sigma \partial_x \zeta_\lambda^\sigma(t,0)                                           & = 0,                               & \text{for} ~ t \in (\tau_1, +\infty),           \\
        \zeta_\lambda^\sigma(\tau_1,x)                                                        & = \zeta^{\sigma, \tau_1}(x),       & \text{for} ~ x \in (0,\be(\tau_1)),
    \end{aligned}
    \right.
\end{equation}
for some $\zeta_\lambda^{\sigma, \tau_1}\in L^2(0, \overline{e}(\tau_1))$ and where $\delta \psi_\lambda^\sigma$ will be defined later in \eqref{eq:feedback}. In particular,
the solution $\zeta_\lambda^\sigma$ to \eqref{linearized_scal} with $\tau_1 = 0$, $\sigma = \sigma_i$, $\delta \psi_\lambda^\sigma = \delta \psi^i$ and $\zeta_\lambda^{\sigma,\tau_1} = z_i^0$ can be identified with $z_i$ the solution to \eqref{linearized_decomp}.

\subsection{Backstepping transformation}

In a nutshell, the general idea of backstepping is to map the original problem \eqref{linearized_scal} to a \emph{target problem} for which exponential or finite-time stability can be proven more easily, and to get the stability result using the reverse transformation. The backstepping approach usually consists in using a ``spatially-causal'' kernel transformation $\mathcal{T}_\lambda^\sigma$, that reads, for any $(t,x)\in U_{\overline{e},\tau_1}$:
\begin{equation}
    \label{eq:map}
    g_\lambda^\sigma(t,x) := \left(\mathcal T_{\lambda}^\sigma \zeta_\lambda^\sigma\right)(t,x),
\end{equation}
where for all $t\geq 0$, all $\xi \in L^2(0, \overline{e}(t))$,
\begin{equation}
    (\mathcal{T}_{\lambda}^\sigma \zeta_\lambda^\sigma)(t,x)=(\mathcal{T}_{\lambda,t}^\sigma \zeta_\lambda^\sigma(t))(x),
\end{equation}
and $\mathcal{T}_{\lambda,t}^\sigma$ is a Volterra transform of the second kind from $L^{2}(0,e(t))$ to itself
\begin{equation}\label{eq:defT}
    \forall x\in (0, \overline{e}(t)), \quad (\mathcal{T}^\sigma_{\lambda,t}\xi)(x) :=  \xi(x) - \int_0^x k^\sigma_\lambda(t;x,y) \xi(y) dy,
\end{equation}
where $k^\sigma_\lambda$ is the solution to the kernel problem \eqref{eq:kernel_eq} which will be introduced below and is a real-valued function defined in the triangular domain
\begin{equation}
    \label{eq:Dt}
    D_{t}:= \left\{ (x,y)\in \left(\mathbb{R}_+\right)^2, \quad  0 < y \leq x < \overline{e}(t)\right\}.
\end{equation}
One of the expected difficulty is that the domain of the problem depends on time and therefore $\mathcal{T}$ and the kernel $k_{\lambda}^{\sigma}$ \emph{a priori} depend on the time $t$. However, an interesting feature of our problem, that we comment about below in Section~\ref{sec:kernels} and Appendix~\ref{app:derivation}, is that the kernel $k_\lambda^\sigma$ actually does not depend on the time $t$ in the sense that it can be chosen as the restriction to $D_t$ of a time-independent function kernel $k^{\sigma,\infty}_\lambda$ defined in a domain
\begin{equation}
    \label{eq:D_infty}
    D_\infty:= \left\{ (x,y)\in \left(\mathbb{R}_+\right)^2, \quad  0 < y \leq x \right\},
\end{equation}
   
namely $k_\lambda^\sigma(t)=k^{\sigma,\infty}_\lambda|_{D_{t}}$ for any $t\geq 0$. Naturally, it holds that for all $0\leq t\leq t'$, $D_t \subset D_{t'} \subset D_\infty$. To alleviate the notations, in the following we will use a slight abuse of notation and denote $k^{\sigma,\infty}_\lambda$ by $k_\lambda^\sigma$.

\medskip

Consequently, we have:
$$
    \forall (t,x)\in U_{\overline{e}}, \quad (\mathcal{T}^\sigma_\lambda w)(t,x) :=  \left( \mathcal T^\sigma_{\lambda,t}w(t)\right)(x) = w(t,x) - \int_0^x k^\sigma_\lambda(x,y) w(t,y) dy,
$$
for any $w \in \left[\mathcal C^0([0,+\infty), L^{2}(0,1))\right]_{\be}$.

\medskip

The main advantage of the transformation $\mathcal T_{\lambda,t}^\sigma$ is that, thanks to the triangular structure, it is always invertible provided that $k_\lambda^\sigma|_{D_t} \in L^2(D_t)$ for all $0\leq t \leq T$ (see Lemma~\ref{lem:inv-map} below). The inverse transformation has then the same form and writes as follows (see Lemma \ref{lem:inv-map}): for any $w \in \left[\mathcal C^0([0,+\infty), L^{2}(0,1))\right]_{\be}$, let us denote by
\begin{equation}\label{eq:inv_relationship}
    \forall (t,x)\in U_{\overline{e}}, \quad  (\mathcal{T}_{\lambda}^{\sigma, {\rm inv}}w)(t,x) := (\mathcal{T}_{\lambda,t}^{\sigma, {\rm inv}}w(t))(x)  = w(t,x) + \int_0^x l^\sigma_{\lambda}(x,y)w(t,y)dy.
\end{equation}
where for all $t\geq 0$ and  all $\xi \in L^2(0, \overline{e}(t))$,
\begin{equation}\label{eq:defTinv}
    \forall x\in (0,\overline{e}(t)), \quad  \left(\mathcal{T}_{\lambda,t}^{\sigma, {\rm inv}}\xi\right)(x) = \xi(x) + \int_0^x l^\sigma_{\lambda}(x,y)\xi(y)dy,
\end{equation}
with $l_\lambda^\sigma$ solution to the inverse kernel problem \eqref{eq:kernel_inv_eq} below. Similarly to $k^\sigma_\lambda$, $l^\sigma_\lambda$ is expected to depend on $t$ but can be chosen as the restriction to $D_t$ of a fixed kernel $l^{\sigma,\infty}_\lambda$ defined in $D_\infty$. In the following we use again the same slight abuse of notation and denote $l^{\sigma,\infty}_\lambda$ by $l^{\sigma}_\lambda$.

We will then see that the following identity holds: for any $t\geq \tau_1$ and $x \in (0,\be(t))$
\begin{equation}
    \label{eq:invmap}
    \zeta_\lambda^\sigma(t,x) = (\mathcal{T}_\lambda^{\sigma, {\rm inv}}g_\lambda^\sigma)(t,x) = \left(\mathcal{T}^{\sigma, {\rm inv}}_{\lambda,t} g_\lambda^\sigma(t)\right)(x),
\end{equation}

Formally, the strategy to identify the set of equations satisfied by $k_\lambda^\sigma$ and $l_\lambda^\sigma$ is to differentiate \eqref{eq:map} in time in space and to write that $\zeta_\lambda^\sigma$ and $g^\sigma_{\lambda}$ must satisfy respectively the initial problem \eqref{linearized_scal} and the target problem \eqref{target} in order to obtain a set of necessary conditions on the kernels $k^\sigma_{\lambda}$ and $l^\sigma_{\lambda}$ (see \eqref{eq:kernel_eq}-\eqref{eq:kernel_inv_eq} below).

\subsection{Target problem}\label{sec:target}

We consider the following target problem:
\begin{equation}
    \label{target}
    \left\{
    \begin{aligned}
        \partial_t g^\sigma_{\lambda} - \sigma \partial_{xx}^2 g^\sigma_{\lambda} + \lambda g^\sigma_{\lambda} & = 0, ~ \text{for} ~ (t,x) \in U_{\be,\tau_1},                         \\
        \sigma \partial_x g^\sigma_{\lambda}(t,\be(t)) +  \bv g^\sigma_{\lambda}(t,\be(t))                     & = 0, ~ \text{for} ~ t \in (\tau_1, +\infty),                          \\
        \sigma \partial_x g^\sigma_{\lambda}(t,0)                                                              & = 0, ~ \text{for} ~ t \in (\tau_1, +\infty),                          \\
        g^\sigma_\lambda(\tau_1,x)                                                                             & = g^{\sigma,\tau_1}_\lambda(x), ~ \text{for} ~ x \in (0,\be(\tau_1)),
    \end{aligned}
    \right.
\end{equation}
that is similar to the original problem \eqref{linearized_scal} but with homogeneous boundary conditions, an additional \emph{damping term} $\lambda g_\lambda^\sigma$ and an initial condition $g_\lambda^{\sigma,\tau_1}\in L^2(0, \be(\tau_1))$.

\medskip

We introduce here a notion of weak $L^2$ solution to~\eqref{target}. To this aim, we introduce the set $D^{\rm targ}$ of test functions $v: U_{\overline{e}, \tau_1} \to \mathbb{R}$ satisfying:
\begin{itemize}
    \label{defDstar-target}
    \item[(i)] $v \in \left[\left(L^2\left((\tau_1,T); H^2\right)\right)\right]_{\overline{e}}\cap \left[\mathcal C^0([\tau_1,T],L^{2})\right]_{\overline{e}}$,
    \item[(ii)] $\partial_t v\in \left[\left(L^2\left((\tau_1,T); L^2\right)\right)\right]_{\overline{e}}$,
    \item[(iii)] $v(T,\cdot) = 0$,
    \item[(iv)] $\sigma \partial_{x}v(t,0) = 0, ~ \mbox{for almost all } t \in (\tau_1,T)$,
    \item[(v)] $\sigma \partial_{x}v(t,\be(t)) = 0, \;  \mbox{for almost all } t \in (\tau_1,T)$.
\end{itemize}

\begin{definition}\label{def:weak-target}
    Let $g_\lambda^{\sigma, \tau_1} \in L^2(0, \overline{e}(\tau_1))$.
    A function $g^\sigma_{\lambda} \in \left[\mathcal C^0([\tau_1, +\infty),L^{2})\right]_{\overline{e}}$ such that $\partial_t g_\lambda^\sigma \in \left[L^2((\tau_1,+\infty);(H^1)')\right]_{\be}$ is said to be a $L^2$-weak solution of \eqref{target} if, for any $T > \tau_1$, it satisfies:

    \begin{equation}
        \label{eq:defatar}
        a^{\rm targ}(g_\lambda^\sigma,v) := \int_{\tau_1}^{T}\int_{0}^{\be(t)} g^\sigma_{\lambda}(t,x) \left[\partial_{t}v(t,x)+\sigma \partial_{xx}^{2}v(t,x) - \lambda v(t,x) \right] dxdt  + \int_{0}^{\be(\tau_1)}g^{\sigma,\tau_1}_{\lambda}(x)  v(\tau_1,x) dx  = 0,
    \end{equation}

    for any test function $v\in D^{\rm targ}$.
\end{definition}

Problem~\eqref{target} is actually exponentially stable with decay rate $\lambda$, that can be chosen arbitrarily large here (see Appendix~\ref{app:target}):
\begin{proposition}[Well-posedness and exponential stability of the target equation]
    \label{prop:stability-target}
    Let $\tau_1\geq 0$, $\sigma, \lambda>0$ and $g_\lambda^{\sigma,\tau_1}\in L^{2}(0,\be(\tau_1))$. Then, there exists a unique weak $L^2$ solution $g^\sigma_{\lambda}\in \mathcal C^0([\tau_1, +\infty), L^{2}(0,\be(t)))$ to \eqref{target} in the sense of Definition~\ref{def:weak-target}, and it holds that, for any $t \geq \tau_1$:
    \begin{equation}
        \label{eq:stabtarget}
        \|g_{\lambda}^\sigma(t)\|_{L^{2}(0,\be(t))} \leq e^{-\lambda (t-\tau_1)}\|g_\lambda^{\sigma,\tau_1}\|_{L^{2}(0,\be(\tau_1))}.
    \end{equation}
\end{proposition}

\subsection{Expression of the feedback and weak solution}
\label{sec:feedback}
Let us first explain here how we can derive an expression for the feedback control $\delta \psi_\lambda^\sigma$. Assume for now $k^\sigma_\lambda$ and $\zeta_\lambda^\sigma$ are smooth and differentiate \eqref{eq:map} with respect to $x$ at $x=\be(t)$. One finds:
\[ \partial_x g^\sigma_{\lambda}(t,\be(t)) = \partial_x \zeta_\lambda^\sigma(t,\be(t)) - k^\sigma_{\lambda}(\be(t),\be(t))\zeta_\lambda^\sigma(t,\be(t)) - \int_0^{\be(t)} \partial_x k^\sigma_{\lambda}(\be(t),y) \zeta_\lambda^\sigma(t,y) dy.   \]
Then, combining with \eqref{eq:map} and considering the second equation of \eqref{target}, namely $\sigma \partial_x g^\sigma_{\lambda}(t,\be(t)) +  \bv g^\sigma_{\lambda}(t,\be(t)) = 0$ and the boundary condition at $x=\be(t)$ in \eqref{linearized_scal}, one must impose the following expression of the feedback, which depends on the kernel $k^\sigma_{\lambda}$: for all $t\geq \tau_1$,

\begin{equation}
    \label{eq:feedback}
    \delta \psi_\lambda^{\sigma}(t) := \sigma k^\sigma_{\lambda}(\be(t),\be(t)) \zeta_\lambda^\sigma(t,\be(t)) + \int_0^{\be(t)} \left[\sigma \partial_x k_{\lambda}(\be(t),y) + \bv k_{\lambda}(\be(t),y)\right] \zeta_\lambda^\sigma(t,y) dy.
\end{equation}

Let us already remark that this feedback is of the form
\begin{equation}\label{eq:feedback-decomp}
    \delta \psi_\lambda^{\sigma}(t) =  H_{l,\lambda}^{\sigma}(t)\zeta^\sigma_\lambda(t) + H_{nl,\lambda}^{\sigma}(t) \zeta_\lambda^\sigma(t),
\end{equation}
where, for any $t \geq 0$ and $\xi \in H^1(0,\be(t))$, the operators are given by
\begin{equation}
    \label{eq:operators}
    \begin{aligned}
        H_{nl,\lambda}^{\sigma}(t)\xi & = \int_0^{\be(t)} \left[\sigma \partial_x k^\sigma_{\lambda}(\be(t),y) + \bv k^\sigma_{\lambda}(\be(t),y)\right] \xi(y) dy, \\
        H_{l,\lambda}^{\sigma}(t) \xi & = \sigma k^\sigma_{\lambda}(\be(t),\be(t)) \xi(\be(t)).
    \end{aligned}
\end{equation}

Assuming now that the feedback is of the form \eqref{eq:feedback-decomp}, we can give a rigorous definition of weak-$L^2$ solutions to problem \eqref{linearized_scal} provided that the family of operators $(H_{nl,\lambda}^{\sigma}(t))_{t \geq 0}$ and $(H_{l,\lambda}^{\sigma}(t))_{t \geq 0}$ satisfy properties (P1')-(P2')-(P3') below, which are scalar versions of properties (P1)-(P2)-(P3).

\medskip

\bfseries Scalar properties of operators:\normalfont

\begin{itemize}\label{properties-operator-scalar}
    \item [(P1')] for almost all $t \geq 0$, $H_{nl,\lambda}^\sigma(t)$ is a continuous linear application from $L^2(0,\overline{e}(t))$ to $\mathbb{R}$;
    \item [(P2')] for all $T>0$, and all $z \in \left[L^2((0,T), L^2)\right]_{\be}$, the application $(0,T)\ni t \mapsto H_{nl,\lambda}^\sigma(t)z(t)$ belongs to $L^2(0,T)$. Moreover, there exists a constant $C=C(T,\sigma,\lambda)>0$ such that
          $$
              \left\|H_{nl,\lambda}^\sigma(\cdot) z(\cdot)\right\|_{\left[L^2((0,T), L^2)\right]_{\be}} \leq C \| z\|_{ \left[L^2((0,T), L^2)\right]_{\be}};
          $$

    \item [(P3')] for almost all $t \geq 0$, the operator $H_{l,\lambda}^\sigma(t): H^1(0, \overline{e}(t)) \to \mathbb{R}$ is defined as follows:
          \begin{equation}\label{eq:Hl-scal}
              \forall z\in  H^1(0, \overline{e}(t)),  \quad H_{l,\lambda}^\sigma(t)z:= K_{l,\lambda}^\sigma(t)z(\overline{e}(t))
          \end{equation}
          where $K_{l,\lambda}^{\sigma} \in L_{loc}^\infty\left(\mathbb{R}_+^*\right)$.
\end{itemize}

We are then in a position to give the definition of weak-$L^2$ solutions to \eqref{linearized_scal}, by analogy with Definition \ref{def:weak-solution}. To this aim, we introduce the set $D^{\rm ini}$ of test functions $w: U_{\overline{e}, \tau_1} \to \mathbb{R}$ satisfying:
\begin{itemize}
    \label{defDpsi-scal}
    \item[(i)] $w \in \left[\left(L^2\left((\tau_1,T); H^2\right)\right)\right]_{\overline{e}}\cap \left[\cC^0([\tau_1,T],L^{2})\right]_{\overline{e}}$,
    \item[(ii)] $\partial_t w\in \left[\left(L^2\left((\tau_1,T); L^2\right)\right)\right]_{\overline{e}}$,
    \item[(iii)] $w(T,\cdot) = 0$,
    \item[(iv)] $\sigma \partial_{x}w(t,0) = 0, ~ \mbox{for almost all }t \in (\tau_1,T)$,
    \item[(v)] $K_l(t)w(t,\overline{e}(t)) - \sigma \partial_{x}w(t,\be(t)) = 0,\;  \mbox{for almost all } t \in (\tau_1,T)$.
\end{itemize}

\begin{definition}[Weak solution in $L^2$ to \eqref{linearized_scal}]
    \label{def:weak-solution3}
    Let $\zeta_\lambda^{\sigma,\tau_1}\in L^2(0, \overline{e}(\tau_1))$. Let $(H_{nl,\lambda}^\sigma(t))_{t \geq 0}$ and $(H_{l,\lambda}^\sigma(t))_{t \geq 0}$ be two families of operators satisfying (P1')-(P2')-(P3'). A function $\zeta_{\lambda}^\sigma \in \left[\cC^0([\tau_1, +\infty),L^{2})\right]_{\overline{e}}$ such that $\partial_t \zeta_\lambda^\sigma \in \left[L^2((\tau_1,+\infty);(H^1)')\right]_{\be}$ is said to be a $L^2$-weak solution to \eqref{linearized_scal} with fluxes $\delta \psi_\lambda^\sigma$ defined by \eqref{eq:feedback-decomp} if, for any $T > \tau_1$, it satisfies:

    \begin{equation}
        \label{eq:defaini}
        \begin{split}
            a^{\rm ini}(\zeta_\lambda^\sigma,w) := \int_{\tau_1}^{T}\int_{0}^{\be(t)} \zeta_\lambda^\sigma(t,x) \left[\partial_{t}w(t,x)+\sigma\partial_{xx}^{2}w(t,x) \right] dxdt \\
            + \int_{0}^{\be(\tau_1)}\zeta_\lambda^{\sigma, \tau_1}(x) w(\tau_1,x) dx
            + \int_{\tau_1}^{T}(H_{nl}(t)\zeta_\lambda^\sigma(t)) w(t,\be(t))dt = 0,
        \end{split}
    \end{equation}

    for any test function $w\in D^{\rm ini}$.
\end{definition}

\subsection{Kernel definition and properties}
\label{sec:kernels}
Now that we have an \emph{a priori} expression for the feedback \eqref{eq:feedback}, it remains to derive the full problems satisfied by the kernels $k^\sigma_{\lambda}$ and $l^\sigma_{\lambda}$. We consider the following problems (recall the definitions of the triangular domains \eqref{eq:Dt} and \eqref{eq:D_infty}):
\begin{equation}
    \label{eq:kernel_eq}
    \left\{
    \begin{aligned}
        \partial_{xx}^2 k^\sigma_{\lambda}(x,y) - \partial_{yy}^2 k^\sigma_{\lambda}(x,y) & = \frac{\lambda}{\sigma}k^\sigma_{\lambda}(x,y) & (x,y) \in D_\infty, \\
        \partial_y k^\sigma_{\lambda}(x,0)                                                & = 0                                             & x \in (0,+\infty),  \\
        k^\sigma_{\lambda}(x,x)                                                           & = -\frac{\lambda}{2 \sigma} x                   & x \in (0,+\infty),
    \end{aligned}
    \right.
\end{equation}

\begin{equation}
    \label{eq:kernel_inv_eq}
    \left\{
    \begin{aligned}
        \partial_{xx}^2 l^\sigma_{\lambda}(x,y) - \partial_{yy}^2 l^\sigma_{\lambda}(x,y) & = -\frac{\lambda}{\sigma}l^\sigma_{\lambda}(x,y) & (x,y) \in D_\infty, \\
        \partial_y l^\sigma_{\lambda}(x,0)                                                & = 0                                              & x \in (0,\infty),   \\
        l^\sigma_{\lambda}(x,x)                                                           & = - \frac{\lambda}{2 \sigma} x            & x \in (0,\infty),
    \end{aligned}
    \right.
\end{equation}
with the notation $\frac{d}{dx} f(x,x) := \partial_x f(x,x) + \partial_y f(x,x)$. It appears that the two problems are related through
\begin{equation}
    \label{eq:kernel-equiv}
    k^\sigma_\lambda = -l^\sigma_{-\lambda}.
\end{equation}
It is rigorously justified below in Lemmas~\ref{lem:map} and~\ref{lem:back-map} that these kernels indeed meet our expectations. Let us however comment here about the derivation of these kernel problems:
\begin{itemize}
    \item First, the derivation is done in Appendix~\ref{app:derivation} assuming that the kernels depend on $t$. In order to explicit the time dependence, one needs to rescale the kernel, which leads to a dynamical boundary problem set in a fixed domain (see \eqref{rescaled-kernel_eq}). Searching for solutions with separate variables, as in \eqref{separate-variables}, one finds the stationary equations \eqref{eq:kernel_eq-Dt}. 
    \item Second, one remarks that any solution to the obtained problem in $D_T$ is in fact a solution to the same problem set in $D_t$ for any $0 \leq t < T$, thanks to the structure of the boundary conditions. Therefore, it suffices to look for a solution in $D_{\infty}$, hence \eqref{eq:kernel_eq}-\eqref{eq:kernel_inv_eq}.
\end{itemize}

Thanks to the structure of the backstepping transformation, we can connect the stability of the two problems: let $g^{\sigma}_\lambda$ be the solution to \eqref{target} in the sense of Proposition~\ref{prop:stability-target}. Then, assuming that \eqref{eq:kernel_inv_eq} has a solution, the function $\zeta_\lambda^{\sigma}$ defined by \eqref{eq:invmap} can be shown to be a solution to \eqref{linearized_scal} (see Lemma~\ref{lem:back-map} below) and it holds that for all $t\geq \tau_1$,

\begin{equation}
    \label{stability_loop}
    \begin{aligned}
        \|\zeta_\lambda^\sigma(t)\|_{L^2(0,\be(t))} & \leq \left(1 + \|l^\sigma_{\lambda}\|_{L^2(D_t)} \right) \|g^\sigma_{\lambda}(t)\|_{L^2(0,\be(t))} \leq \left(1 + \|l^\sigma_{\lambda}\|_{L^2(D_t)}\right) e^{-\lambda (t-\tau_1)} \|g_\lambda^{\sigma, \tau_1}\|_{L^2(0,\be(\tau_1))}, \\
                                                    & \leq \left(1 + \|l^\sigma_{\lambda}\|_{L^2(D_t)}\right) \left(1 + \|k^\sigma_{\lambda}\|_{L^2(D_{\tau_1})} \right) e^{-\lambda (t-\tau_1)} \|\zeta^{\sigma,\tau_1}\|_{L^2(0,\be(\tau_1))}.
    \end{aligned}
\end{equation}
Hence, to get the desired stability, the remaining key point of the analysis is the control of $\|l^\sigma_{\lambda}\|_{L^2(D_t)}$ with respect to time. For this, we study the following problem, for $\alpha \in \R$,
    \begin{equation}
        \label{eq:kpde}
        \left\{
        \begin{aligned}
            \partial_{xx}^2 k^\alpha(x,y) - \partial_{yy}^2 k^\alpha(x,y) & = \alpha k^{\alpha}(x,y) & (x,y) \in D_\infty, \\
            \partial_y k^{\alpha}(x,0)                                    & = 0                      & x \in (0,\infty),   \\
            k^\alpha(x,x)                                                 & = -\frac{\alpha}{2} x    & x \in (0,\infty),   \\
        \end{aligned}
        \right.
    \end{equation}
of which \eqref{eq:kernel_eq} and \eqref{eq:kernel_inv_eq} are instances. We consider the following definition of weak solution to \eqref{eq:kpde}:
    \begin{definition}\label{def:defkernel}
        A function $k^\alpha: D_\infty \to \mathbb{R}$ is said to be a weak solution to \eqref{eq:kpde} if and only if the two following conditions are satisfied:
        \begin{itemize}
            \item [(i)]
                  the function $\overline{k}^\alpha: (0, +\infty)^2 \to \mathbb{R}$ defined such that
                  $$
                      \overline{k}^\alpha(x,y):= \left\{
                      \begin{array}{ll}
                          k^\alpha(x,y)       & \mbox{ if } (x,y)\in D_\infty:= \{0 < y \leq x < \infty \}, \\
                          - \frac{\alpha}{2}x & \mbox{otherwise}                                            \\
                      \end{array}\right.
                  $$
                  is such that, for any $L>0$, the function $\overline{k}^\alpha_L:= \overline{k}^\alpha|_{[0,L]^2}$ is such that $\overline{k}^\alpha_L \in \mathcal C^0([0,L], H^1(0,L))$, $\partial_x \overline{k}^\alpha_L \in \mathcal C^0([0,L], L^2(0,L))$ and $\partial_{xx} \overline{k}^\alpha_L \in \mathcal C^0([0,L], H^1(0,L)')$;
            \item[(ii)] for all $L>0$ and for all $v,w\in H^1(0,L)$,
                  \begin{align*}
                       & - \int_0^L \left(\int_0^x \partial_{x} k^\alpha(x,y)v(y)\,dy \right) \partial_x w(x)\,dx  + w(L)\int_0^L \partial_x k^\alpha(L,y)v(y)\,dy \\
                       & + \int_0^L \int_0^x \partial_y k^\alpha(x,y) \partial_y v(y)\,dy  w(x)\,dx                                                                \\
                       & = \alpha
                      \int_0^L \left(\int_0^x k^\alpha(x,y) v(y)\,dy\right) w(x)\,dx                                                                               \\
                       & + \frac{\alpha}{2}\int_0^L  v(x) w(x)\,dx.                                                                                                \\
                  \end{align*}
        \end{itemize}
    \end{definition}

Well-posedness and estimates for \eqref{eq:kernel_eq} and \eqref{eq:kernel_inv_eq} are achieved in the following proposition, which is proven in Section \ref{proof:prop3}.

\medskip

\begin{proposition}
    \label{prop:estimates}
    Let $\sigma> 0$. For any $\lambda \geq 0$, there exists a unique weak solution $k^\sigma_\lambda$ (resp. $l^\sigma_\lambda$) (in the sense of Definition~\ref{def:defkernel}) to the kernel problem \eqref{eq:kernel_eq} (resp. \eqref{eq:kernel_inv_eq}). Moreover, there exist $\lambda_\sigma >0$ and constants $C,c>0$ independent of $\sigma$ such that, for any $\lambda \geq \lambda_\sigma$, $t \geq 0$ and $x \in (0,\be(t))$:
    \begin{equation}
        \label{k-estimate}
        \int_0^x \left(|k^\sigma_{\lambda}(x,y)|^2 + |\nabla k^\sigma_{\lambda}(x,y)|^2 \right)\,dy \leq C e^{c \be(t)\sqrt{\lambda/\sigma}},
    \end{equation}
    \begin{equation}
        \label{l-estimate}
        \int_0^x \left(|l^\sigma_{\lambda}(x,y)|^2 + |\nabla l^\sigma_{\lambda}(x,y)|^2 \right)\,dy \leq C \left(\frac{\lambda}{\sigma}\right)^4 e^{c \be(t)}.
    \end{equation}
\end{proposition}

\begin{remark}
    An immediate consequence of \eqref{k-estimate}-\eqref{l-estimate} is that for any $t\geq 0$, $k_\lambda^\sigma|_{D_t}\in H^1(D_t)$, $l_\lambda^\sigma|_{D_t}\in H^1(D_t)$ and
    \begin{equation}
        \label{eq:estimate-k}
        \|k^\sigma_{\lambda}\|^2_{H^1(D_t)} \leq C e^{\tilde{c}\be(t)\sqrt{\lambda/\sigma}} ,
    \end{equation}
    \begin{equation}
        \label{eq:estimate-l}
        \|l^\sigma_{\lambda}\|^2_{H^1(D_t)} \leq C  \left(\frac{\lambda}{\sigma}\right)^2 e^{\tilde{c}\be(t)}.
    \end{equation}
\end{remark}

\begin{remark}
    It was shown in (\cite{liu2003boundary}, Lemma 3.2) that the kernel solutions obtained in Proposition~\ref{prop:estimates} are more regular, namely $C^2$. The proof is based on an integral reformulation and a series representation formula. We have chosen to adopt a weak framework here since on the one hand, it is an appropriate framework to derive estimates \eqref{eq:estimate-k}-\eqref{eq:estimate-l} and on the other hand, it shows that our strategy can be extended to equations with space-dependent coefficients \cite{coron2017}.
\end{remark}

\subsection{Main auxiliary results}
\label{sec:auxiliary}

Next, we check that the kernels defined as solutions to \eqref{eq:kernel_eq}-\eqref{eq:kernel_inv_eq} indeed map \eqref{linearized_scal} to \eqref{target} through the transformation $\mathcal{T}^\sigma_\lambda$ and the other way around. In fact, we need to check that the formal computations performed for the derivation of the kernel problems can be adapted to the case when we have to consider weak $L^2$ solutions in the sense of Definitions~\ref{def:weak-solution} and \ref{def:weak-target}.

We start by stating in Lemma~\ref{lem:inv-map} that for all $\lambda,\sigma>0$ and $t\geq 0$, the transformation $\mathcal{T}_{\lambda,t}^\sigma: L^2(0, \be(t)) \to L^2(0, \be(t))$ associated to the unique kernel solution to \eqref{eq:kernel_eq} is one-to-one, and that it can be inverted from $L^2(0,\be(t))$ to its image, with inverse given by $\mathcal{T}_{\lambda,t}^{\sigma,{\rm inv}}$. The proof of Lemma~\ref{lem:inv-map} in the case of Dirichlet boundary conditions is provided in \cite{coron2017} (Lemma 4). We omit the proof here, which is very similar. We also refer to \cite{liu2003boundary} for the invertibility of the transformation with Neumann boundary conditions (Lemma 3.3).
\begin{lemma}
    \label{lem:inv-map}
    Let $\lambda, \sigma>0$ and $t \geq 0$. Then,
    $ \mathcal T_{\lambda,t}^{\sigma, {\rm inv}} \circ \mathcal T_{\lambda,t}^{\sigma} = \mathcal T_{\lambda,t}^{\sigma} \circ \mathcal T_{\lambda,t}^{\sigma, {\rm inv}} = {\rm Id}_{L^2(0, \be(t))}$.
\end{lemma}

Then in Lemma~\ref{lem:map}, we check that this transformation indeed transforms \eqref{linearized_scal} into \eqref{target} when the boundary term $\delta \psi_\lambda^{\sigma}$ in \eqref{linearized_scal} is defined in \eqref{eq:feedback-decomp}-\eqref{eq:operators}.

\begin{lemma}
    \label{lem:map}
    Let $\sigma >0$, let $\lambda \geq \lambda_\sigma$ where $\lambda_\sigma$ is defined in Proposition~\ref{prop:estimates}, $\tau_1 \geq 0$ and $\zeta_\lambda^{\sigma,\tau_1}\in L^2(0, \overline{e}(\tau_1))$. Let $k^\sigma_\lambda$ be the unique weak solution to \eqref{eq:kernel_eq} in the sense of Definition~\ref{def:defkernel}. Assume that $\zeta^\sigma_{\lambda}$ is a weak-$L^2$ solution to \eqref{linearized_scal} in the sense of Definition~\ref{def:weak-solution3} where $H_{l,\lambda}^{\sigma}$ and $H_{nl,\lambda}^{\sigma}$ are defined by \eqref{eq:operators} and $\delta \psi_\lambda$ by \eqref{eq:feedback-decomp}. For all $t \geq \tau_1$, define $g^\sigma_\lambda(t):= \mathcal T^\sigma_{\lambda,t}\zeta^\sigma_{\lambda}(t)$ where $\mathcal T_{\lambda,t}^\sigma$ is defined by \eqref{eq:defT}. Then $g^\sigma_\lambda$ is the unique weak $L^2$ solution to \eqref{target} in the sense of Definition~\ref{def:weak-target} with $g_\lambda^{\sigma, \tau_1}= \mathcal T^\sigma_{\lambda,\tau_1}\zeta^{\sigma,\tau_1}_{\lambda}$.
\end{lemma}

The objective of Lemma~\ref{lem:back-map} is to state the following point: let $g^\sigma_{\lambda}$ be the solution to the target problem, then $\zeta^\sigma_{\lambda} := \mathcal{T}_{\lambda}^{\sigma, {\rm inv}} g^\sigma_{\lambda}$ is a solution to the original problem.

\begin{lemma}\label{lem:back-map}
    Let $\sigma >0$, let $\lambda \geq \lambda_\sigma$ where $\lambda_\sigma$ is defined in Proposition~\ref{prop:estimates}, $\tau_1 \geq 0$ and $g_{\lambda}^{\sigma,\tau_1} \in L^2(0,\be(\tau_1))$. Let $l^\sigma_\lambda$ be the unique weak solution to \eqref{eq:kernel_inv_eq}. Let $g^\sigma_{\lambda}$ be the unique weak-$L^2$ solution to \eqref{target} in the sense of Proposition~\ref{prop:stability-target}.

    For all $t \geq \tau_1$, define $\zeta^\sigma_\lambda(t):= \mathcal T_{\lambda, t}^{\sigma, {\rm inv}} g_\lambda^\sigma(t)$ where $\mathcal T_{\lambda,t}^{\sigma, {\rm inv}}$ is defined in \eqref{eq:defTinv}. Then, $\zeta_{\lambda}$ is a weak-$L^2$ solution to \eqref{linearized_scal} with $\zeta_\lambda^{\sigma, \tau_1} = \mathcal T_{\lambda,\tau_1}^{\sigma, \rm inv} g_\lambda^{\sigma,\tau_1}$, $\delta \psi_\lambda^\sigma$ defined by \eqref{eq:feedback-decomp} and $(H_{l,\lambda}^{\sigma}(t))_{t \geq 0}, (H_{nl,\lambda}^{\sigma}(t))_{t \geq 0}$ defined by \eqref{eq:operators}.

\end{lemma}

The proofs of Lemmas~\ref{lem:map} and \ref{lem:back-map} are postponed to Section~\ref{proof:lems}. Lemma~\ref{lem:back-map} together with Propositions~\ref{prop:stability-target} and \ref{prop:estimates} yield the existence of at least one weak-$L^2$ solution to \eqref{linearized_scal} for any $\zeta_{\lambda}^{\sigma,\tau_1}\in L^2(0, \be(\tau_1))$ provided that $\lambda \geq \lambda_\sigma$, and this solution satisfies the stability estimate \eqref{eq:stability}.  Lemmas~\ref{lem:inv-map} and \ref{lem:back-map} yield uniqueness of this solution. As a consequence, the feedback control \eqref{eq:feedback} stabilizes \eqref{linearized_scal} exponentially with an arbitrary decay provided $\lambda$ is chosen large enough. The object of Corollary~\ref{cor:exp-stability} is to summarize these points.

\begin{corollary}\label{cor:exp-stability}
    Let $\sigma >0$, $\tau_1\geq 0$, $\lambda \geq \lambda_\sigma$ where $\lambda_\sigma$ is defined in Proposition~\ref{prop:estimates} and $\zeta_\lambda^{\sigma,\tau_1} \in L^2(0, \be(\tau_1))$. Then, there exists a unique weak-$L^2$ solution to problem \eqref{linearized_scal} in the sense of Definition~\ref{def:weak-solution3} with $\delta \psi_\lambda^\sigma$ defined by \eqref{eq:feedback-decomp} and operators $(H_{nl,\lambda}^\sigma(t))_{t \geq \tau_1}$ and $(H_{l,\lambda}^\sigma(t))_{t \geq \tau_1}$ defined by \eqref{eq:operators}. Moreover, there exist constants $C,c>0$ independent of  $\lambda$, $\sigma$ and $t$ such that this solution satisfies, for any $t\geq \tau_1$:
    \begin{equation}
        \label{eq:stability}
        \|\zeta_{\lambda}^\sigma(t)\|_{L^2(0,\be(t))} \leq C \left(1 + \left(\frac{\lambda}{\sigma}\right)^2 \right) e^{c \be(\tau_1) \sqrt{\lambda/\sigma} + c \be(t) - \lambda (t-\tau_1)} \|\zeta_\lambda^{\sigma,\tau_1} \|_{L^2(0,\be(\tau_1))}.
    \end{equation}
\end{corollary}

\begin{proof}
    \emph{Existence and estimate} : Let $k_{\lambda}^\sigma$ and $l_{\lambda}^\sigma$ be the kernels defined in Proposition~\ref{prop:estimates}. Since $k_{\lambda}^\sigma|_{D_{\tau_1}} \in L^2(D_{\tau_1})$, one can define $g_{\lambda}^{\sigma,\tau_1} := \mathcal{T}_{\lambda,\tau_1}^\sigma \zeta_{\lambda}^{\sigma,\tau_1} \in L^2((0,\be(\tau_1)))$. Then by Proposition~\ref{prop:stability-target}, there exists a unique weak-$L^2$ solution $g_{\lambda}^\sigma$ to \eqref{target} with initial condition $g_{\lambda}^{\sigma,\tau_1}$, and this solution satisfies \eqref{eq:stabtarget}. Since, for any $t \geq \tau_1$, $l_{\lambda}^\sigma|_{D_t} \in L^2(D_t)$, one can define $\zeta_{\lambda}^\sigma(t) := \mathcal{T}_{\lambda,t}^{\sigma,\rm inv} g_{\lambda}^{\sigma}(t)$ and by Lemma~\ref{lem:back-map}, it defines a solution to \eqref{linearized_scal} associated to operators $(H_{nl,\lambda}^\sigma(t))_{t \geq \tau_1}$ and $(H_{l,\lambda}^\sigma(t))_{t \geq \tau_1}$. Moreover, estimate \eqref{eq:stability} follows from the definition of $\mathcal{T}_{\lambda}^{\sigma,\rm inv}$ together with the estimates \eqref{eq:stabtarget} and \eqref{l-estimate}:
    \begin{equation*}
        \begin{aligned}
            \|\zeta_\lambda^\sigma(t)\|_{L^2(0,\be(t))} & \leq \left(1 + \|l^\sigma_{\lambda}(t)\|_{L^2(D_t)}\right) \left(1 + \|k^\sigma_{\lambda}(\tau_1)\|_{L^2(D_{\tau_1})} \right) e^{-\lambda (t-\tau_1)} \|\zeta^{\sigma,\tau_1}\|_{L^2(0,\be(\tau_1))}                  \\
                                                        & \leq \left(1 + C \left(\frac{\lambda}{\sigma}\right)^2 e^{c\be(t)} \right) \left(1 + Ce^{c\be(\tau_1)\sqrt{\lambda/\sigma}} \right) e^{-\lambda(t-\tau_1)} \|\zeta_{\lambda}^{\sigma,\tau_1} \|_{L^2(0,\be(\tau_1))}, \\
                                                        & \leq \tilde{C} \left(1 + \left(\frac{\lambda}{\sigma}\right)^2 \right) e^{c \be(\tau_1) \sqrt{\lambda/\sigma} + c \be(t) - \lambda (t-\tau_1)} \|\zeta_\lambda^{\sigma,\tau_1} \|_{L^2(0,\be(\tau_1))}
        \end{aligned}
    \end{equation*}

    \emph{Uniqueness} : take two weak-$L^2$ solutions to \eqref{linearized_scal} $\zeta_1$ and $\zeta_2$. Then by Lemma~\ref{lem:map}, it holds:
    \[ \mathcal{T_{\lambda}^\sigma}(\zeta_1-\zeta_2) = 0, \]
    but Lemma~\ref{lem:inv-map} yields:
    \[ \zeta_1=\zeta_2.\]

\end{proof}

\section{Proofs}\label{sec:proofs}

\subsection{Proof of Proposition~\ref{prop:estimates}}\label{proof:prop3}

We begin by proving a few preliminary lemmas. In the following, the variable $x$ should be interpreted as the time variable of a wave equation.

\begin{lemma}
    \label{lem:square}
    Let $\alpha\in \mathbb{R}$, $L >0$ and $f \in L^2((0,L)^2)$. Then, there exists a unique solution $K \in  \cC^0([0,L], H^1(0,L))$ such that $\partial_x K \in \cC^0([0,L], L^2(0,L))$ and $\partial_{xx} K \in L^2((0,L), (H^1(0,L))')$ solution to the equation
    \begin{equation}
        \label{eq:complete}
        \left\{ \begin{array}{ll}
            \partial_{xx}K(x,y) - \partial_{yy}K(x,y)  =  \alpha K(x,y) + f(x,y), & ~ \text{for} ~ (x,y)\in(0,L)^2, \\
            \partial_y K(x,0)=\partial_y K(x,L) = 0 ,                             & ~ \text{for} ~ x \in (0,L),     \\
            K(0,y) = \partial_x K(0,y) = 0,                                       & ~ \text{for} ~ y \in (0,L),     \\
        \end{array}
        \right.
    \end{equation}
    in the sense that, for all $v\in H^1(0,L)$, for almost all $x\in (0,L)$,
    \begin{equation}\label{eq:weakwave}
        \left\langle \partial_{xx}K(x,\cdot),v\right\rangle_{H^1(0,L)', H^1(0,L)} + \int_0^L \partial_y K(x,y) \partial_y v(y)\,dy = \alpha \int_0^L K(x,y)v(y)\,dy + \int_0^L f(x,y)v(y)\,dy.
    \end{equation}

    Moreover, there exists a constant $C > 0$ independent of $\alpha$ and $L$ such that for almost any $x \in (0,L)$,
    \begin{equation}
        \label{eq:H1-estimate}
        \int_0^L \left(|K(x,y)|^2 + |\nabla K(x,y)|^2\right) dy \leq (1 + L^2)e^{C\max([\alpha]_+^{1/2}, 1)L} \|f\|_{L^2((0,L)^2)}^2,
    \end{equation}
    where $[\alpha]_+:= \max(\alpha, 0)$ denotes the positive part of $\alpha$.
\end{lemma}

\begin{proof}
    The existence and uniqueness of a solution $K$ to problem~\eqref{eq:complete} in the sense of \eqref{eq:weakwave} such that $K \in  \cC^0([0,L], H^1(0,L))$, $\partial_x K \in \cC^0([0,L], L^2(0,L))$ and $\partial_{xx} K \in L^2((0,L), (H^1(0,L))')$ is a direct consequence of~\cite{brezis_functional_2011}[Theorem~10.14,p.345]. Let us now prove estimate \eqref{eq:H1-estimate}.

    \medskip

    {\bfseries Step~1 (smooth $f$):} Let us first assume that $f$ satisfies the additional regularity constraint $\partial_x f \in L^2((0,L)^2)$. Then differentiating the equation with respect to $x$ as in the proof of \cite{evans2010}[Theorem~5,p.389], it can be checked that $\partial_{xx}K \in L^\infty((0,L), L^2(0,L))$,
    $\partial_x K \in L^\infty((0,L), H^1(0,L))$ and $K\in L^\infty((0,L), H^2(0,L))$. In particular, for almost all $x\in (0,L)$, $\partial_{xx}K(x,\cdot)\in L^2(0,L)$, $\partial_x K(x,\cdot)\in H^1(0,L)$ and $K(x,\cdot)\in H^2(0,L)$.

    Taking $v=\partial_x K(x,\cdot)$ as a test function in \eqref{eq:weakwave} yields that for almost all $x\in (0,L)$,
    \begin{align*}
        \left\langle \partial_{xx}K(x,\cdot),\partial_x K(x,\cdot)\right\rangle_{H^1(0,L)', H^1(0,L)} + \int_0^L \partial_y K(x,y) \partial_{xy} K(x,y) \,dy & = \alpha \int_0^L K(x,y) \partial_x K(x,y)\,dy \\
                                                                                                                                                             & + \int_0^L f(x,y)\partial_x K(x,y)\,dy.        \\
    \end{align*}
    This yields, using the Aubin-Lions theorem, that
    \begin{equation}
        \label{eq:weak-K}
        \frac{1}{2} \frac{d}{dx}\left( \int_0^L \left(|\partial_{x}K(x,y)|^2 + |\partial_{y}K(x,y)|^2\right)\,dy \right) = \frac{\alpha}{2}\frac{d}{dx}\left(  \int_0^L |K(x,y)|^2\,dy\right) + \int_0^L f(x,y)\partial_x K(x,y)\,dy.
    \end{equation}

    Now, using the fact that $\partial_y K(0,y) = 0$ (since $K(0,y) = 0$ for almost all $y\in (0,L)$) and the fact that $f(x,y)\partial_xK (x,y) \leq \frac{1}{2} (f(x,y)^2 + \partial_x K(x,y)^2)$ and integrating \eqref{eq:weak-K} between $0$ and $x$, we obtain that for almost all $x \in (0,L)$:
    
    \begin{equation}
        \label{eq:first-estimate}
        \int_0^L \left[ |\partial_x K(x,y)|^2 + |\partial_y K(x,y)|^2 \right] dy \leq \alpha \int_0^L |K(x,y)|^2 dy + \int_0^x \int_0^L \left[ |\partial_x K(s,y)|^2 + |\partial_y K(s,y)|^2 \right] \,dy\,ds + \|f\|_{L^2((0,L)^2)}^2. 
    \end{equation} 

    \bfseries Case~1: \normalfont If $\alpha \leq 0$, an immediate Gronwall argument yields that for almost all $x \in (0,L)$:
    \begin{equation}
        \label{eq:case1}
        \int_0^L \left[ |\partial_x K(x,y)|^2 + |\partial_y K(x,y)|^2 \right] dy \leq e^x \|f\|_{L^2((0,L)^2)}^2 \leq e^L \|f\|_{L^2((0,L)^2)}^2.
    \end{equation}

    \bfseries Case~2: \normalfont If $\alpha >0$, we perform a change of variables: let us define, for all $\hat{x}\in (0, \alpha^{1/2}L)$,
    \[\hat{K}(\hat{x},y) = K(\alpha^{-1/2}\hat{x},y),\]
    such that 
    \[\partial_{x} \hat{K}(\hat{x}, y) = \alpha^{-1/2}\partial_x K(\alpha^{-1/2}\hat{x},y).\]
    Then for all $\hat{x} \in (0,\alpha^{1/2}L)$, rewrite \eqref{eq:first-estimate} with $x = \alpha^{-1/2} \hat{x}$ as
    \[
    \begin{split} \int_0^L \left[ |\partial_x K(\alpha^{-1/2} \hat{x},y)|^2 + |\partial_y K(\alpha^{-1/2} \hat{x},y)|^2 \right] dy \leq \alpha \int_0^L |K(\alpha^{-1/2} \hat{x},y)|^2 dy \\ 
    + \int_0^{\alpha^{-1/2} \hat{x}} \int_0^L \left[ |\partial_x K(s,y)|^2 + |\partial_y K(s,y)|^2 \right] \,dy\,ds    + \|f\|_{L^2((0,L)^2)}^2. 
    \end{split}
    \]
    Performing the change of variable $\hat{s} = \alpha^{1/2} s$ in the double integral on the right-hand side, and dividing everything by $\alpha$, one gets   
    \begin{equation}
        \label{eq:estimate-scaling}
        \begin{split}
            \int_0^L \left[ \left|\partial_{x} \hat{K}(\hat{x},y)\right|^2 + \alpha^{-1}\left|\partial_y \hat{K}(\hat{x},y)\right|^2 \right] dy & \leq \int_0^L \left|\hat{K}(\hat{x},y)\right|^2 dy +\\
            & \alpha^{-1/2}\int_0^{\hat{x}} \int_0^L \left[\left|\partial_x\hat{K}(\hat{s},y)\right|^2 + \alpha^{-1}\left|\partial_y \hat{K}(\hat{s},y)\right|^2\right] \,dy\,d\hat{s}\\
            +& \alpha^{-1}\|f\|_{L^2((0,L)^2)}^2.
        \end{split}
    \end{equation}
    Let us define for almost all $\hat{x} \in (0,\alpha^{1/2} L)$
    \begin{equation}
        \label{defV1V2}
        V_1(\hat{x}) = \int_0^L \left[ \left|\partial_x\hat{K}(\hat{x},y)\right|^2 + \alpha^{-1} \left|\partial_y\hat{K}(\hat{x},y)\right|^2 \right]dy ~ \text{and}~ V_2(\hat{x})=\int_0^L \left|\hat{K}(\hat{x},y)\right|^2dy.
    \end{equation}
    The previous estimate can be equivalently rewritten as: for almost all $\hat{x} \in (0,\alpha^{1/2} L)$
    \[ V_1(\hat{x}) \leq V_2(\hat{x}) + \alpha^{-1/2} \int_0^{\hat{x}} V_1(s)\,ds + \alpha^{-1} \|f\|_{L^2}^2. \]
    Notice also that
    \[V_2'(\hat{x})= 2 \int_0^L \partial_x\hat{K}(\hat{x},y) \hat{K}(\hat{x},y) dy \leq 2 V_1(\hat{x})^{1/2} V_2(\hat{x})^{1/2} \leq V_1(\hat{x}) +  V_2(\hat{x}), \]
    so that
    \[ V_1(\hat{x}) + V_2'(\hat{x}) \leq 3V_2(\hat{x}) + 2 \alpha^{-1/2} \int_0^{\hat{x}} V_1(s)ds + 2 \alpha^{-1} \|f\|_{L^2}^2. \]
    We are now in the position to use a Gronwall-type argument. Set $g(\hat{x}) := \int_0^{\hat{x}} V_1(s)\,ds + V_2(\hat{x})$. The previous estimate then reads as: for almost any $\hat{x} \in (0,\alpha^{1/2} L)$
    \[ g'(\hat{x}) \leq C_{\alpha} g(\hat{x}) + D_{\alpha} \|f\|_{L^2}^2. \]
    with
    \begin{align*}
        C_{\alpha} & := \max(3,2\alpha^{-1/2}), \\
        D_{\alpha} & := 2 \alpha^{-1}.
    \end{align*}
    Therefore,
    \begin{equation}
        \label{estimg}
        g(\hat{x}) \leq D_{\alpha} \|f\|_{L^2}^2 \hat{x} e^{C_{\alpha} \hat{x}}.
    \end{equation}

    Now rewrite \eqref{eq:estimate-scaling} in terms of $V_1,V_2$: for almost any $\hat{x} \in (0,\alpha^{1/2} L)$
    \begin{equation}
        \label{eq:exprV1V2}
        V_1(\hat{x}) = \int_0^L \left[ \left|\partial_x\hat{K}(\hat{x},y)\right|^2 + \alpha^{-1}\left| \partial_y K(\hat{x},y)\right|^2 \right] dy \leq V_2(\hat{x}) + \alpha^{-1/2} \int_0^{\hat{x}} V_1(s)\,ds + \alpha^{-1} \|f\|_{L^2}^2.
    \end{equation}
    Thus,
    \begin{align*}
        \alpha^{-1} \int_0^L \left[ |\partial_x K(\hat{x},y)|^2 + |\partial_y K(\hat{x},y)|^2 \right] dy & = V_1(\hat{x})                                                                                                    \\
                                                                                                         & \leq \max(1, \alpha^{-1/2}) g(\hat{x}) + \alpha^{-1}\|f\|_{L^2}^2                                                 \\
                                                                                                         & \leq \alpha^{-1}\|f\|_{L^2}^2\left( 1 + 2 \max(1, \alpha^{-1/2}) \hat{x} e^{C_\alpha \hat{x}} \right)             \\
                                                                                                         & \leq \alpha^{-1}\|f\|^2_{L^2}\left( 1 + 2 \max(\alpha^{1/2}, 1)L e^{\max(3 \alpha^{1/2},2)L} \right)              \\
                                                                                                         & \leq \alpha^{-1}\|f\|^2_{L^2}\left( \left(1 + 2 \max(\alpha^{1/2}, 1)L \right)e^{\max(3 \alpha^{1/2},2)L} \right) \\
                                                                                                         & \leq \alpha^{-1}\|f\|^2_{L^2} e^{3\max(3 \alpha^{1/2},2)L}.                                                       \\
    \end{align*}
    We thus finally obtain that for almost all $x\in (0,L)$
    \begin{equation}
        \label{eq:case2}
        \int_0^L \left[ |\partial_x K(x,y)|^2 + |\partial_y K(x,y)|^2 \right] dy \leq \|f\|^2_{L^2} e^{\max(6 \alpha^{1/2},4)L}.
    \end{equation}

    Therefore, combining \eqref{eq:case1} and \eqref{eq:case2} we have proven so far that, for $\alpha \in \R$,
    \begin{equation}
        \label{eq:gradient-estimates}
        \int_0^L \left[ |\partial_x K(x,y)|^2 + |\partial_y K(x,y)|^2 \right] dy \leq e^{C\max([\alpha]_+^{1/2}, 1)L} \|f\|_{L^2((0,L)^2)}^2.
    \end{equation}

    Thanks to the null initial conditions ($x=0$) it holds for almost any $(x,y) \in (0,L)^2$
    \[ K(x,y) = \int_0^x \partial_x K(z,y) dz \leq \sqrt{L} \sqrt{\int_0^L |\partial_x K(z,y)|^2 dz} \]
    Integrate over $y \in (0,L)$ the square of this inequality: for almost any $x \in (0,L)$
    \begin{equation}
        \label{eq:L^2}
        \int_0^L |K(x,y)|^2 dy \leq L \int_0^L \int_0^L |\partial_x K(z,y)|^2 dy dz \leq  L^2 e^{C\max([\alpha]_+^{1/2}, 1)L} \|f\|_{L^2((0,L)^2)}^2,
    \end{equation}
    where we used \eqref{eq:gradient-estimates} for the last inequality. Hence the result when $f$ is a smooth function.

    \medskip

    {\bfseries Step~2 (approximation):} Let us now turn to the case when $f\in L^2((0,L)^2)$. Then, there exists a sequence $(f_n)_{n\in \mathbb{N}}$ of functions in $C^{\infty}_c((0,L)^2)$ such that $\|f_n-f\|_{L^2((0,L)^2)} \xrightarrow{n \to \infty} 0$. Let us denote by $K_n$ the unique solution of \eqref{eq:complete} with $f=f_n$  for all $n\in \mathbb{N}$. By standard results on the wave equation (see \cite{evans2010}[Theorem 5,p.410]), there exists a constant $C>0$ independent of $n$ such that for all $n\in \mathbb{N}$
    $$
        \|K_n - K\|_{L^\infty( (0,L), H^1(0,L))} + \|\partial_x K_n - \partial_x K\|_{L^\infty((0,L), L^2(0,L))} \leq C \|f_n -f \|_{L^2((0,L))}.
    $$
    Thus, passing to the limit $n\to +\infty$ in the inequality
    $$
        \int_0^L \left(|K_n(x,y)|^2 + |\nabla K_n(x,y)|^2\right) dy \leq (1 + L^2)e^{C\max([\alpha]_+^{1/2}, 1)L} \|f_n\|_{L^2((0,L)^2)}^2,
    $$
    which holds for almost all $x\in (0,L)$ yields the desired result.

\end{proof}

\begin{lemma}
    \label{lem:localization}
    In the framework of Lemma~\ref{lem:square}, assume in addition that in \eqref{eq:complete}, $${\rm  Supp} ~ f \subset D_L:=\left\{ (x,y)\in (0,L)^2, \; 0<y \leq x <L\right\}.$$ Then it holds that
    \begin{equation}
        \label{localization}
        K(x,y) = 0 ~ \text{a.e in} ~ (0,L)^2 ~\backslash~ D_L.
    \end{equation}
\end{lemma}
\begin{proof}

    Consider the restriction of the $H^1$ energy: for almost any $x \in (0,L)$,
    \[ E(x) := \frac{1}{2} \int_x^L \left(K^2(x,y) + (\partial_x K)^2(x,y) + (\partial_y K)^2(x,y)\right) dy.\]

    Assume first that $f$ is smooth in the sense that $\partial_x f \in L^2((0,L)^2)$.

    Then, the function $E$ is absolutely continuous, and it holds that
    \begin{align}
        \label{eq:diff-ener}
        E'(x) & = \int_x^L \left[ \partial_xK(x,y)K(x,y) + \partial_{xx}K(x,y) \partial_xK(x,y) + \partial_{xy}K \partial_y K(x,y)\right]\,dy \\ \nonumber
              & - \frac{1}{2}\left(K^2(x,x)+(\partial_xK)^2(x,x) + (\partial_y K)^2(x,x) \right).                                             \\ \nonumber
    \end{align}
    Integrating by parts the last term yields
    \[ \int_x^L (\partial_{xy}K(x,y) \partial_y K(x,y)) dy = -\int_x^L \partial_x K(x,y) \partial_{yy}K(x,y)\,dy - \partial_x K(x,x)\partial_y K(x,x).\]
    Using the fact that $K$ is a solution of \eqref{eq:complete}, and that ${\rm Supp} ~ f \subset D_{L}=\{0 < y \leq x < L \}$, we obtain:
    \begin{align*}
        E'(x) & = (\alpha+1)\int_x^L \partial_x K(x,y) K(x,y) dy - \frac{1}{2} K^2(x,x) - \frac{1}{2} \left((\partial_x K)^2(x,x) + (\partial_y K)^2(x,x) + 2 \partial_x K(x,x) \partial_y K(x,x)\right) \\
              & \leq (\alpha+1)\int_x^L \partial_x K(x,y) K(x,y) dy                                                                                                                                      \\
              & = -(\alpha+1) \int_x^L \partial_x K(x,y) \int_y^L \partial_y K(x,s) \,ds \,dy + (\alpha+1) K(x,L) \int_x^L \partial_x K(x,y)dy.
    \end{align*}
    The Cauchy-Schwarz inequality enables to bound the first term by $2|\alpha+1|(L-x)E(x)$ and the second term by $|\alpha+1||K(x,L)|\sqrt{L-x} (2E(x))^{1/2}$. Then, we use the one-dimensional Sobolev inequality on $y\rightarrow K(x,y)$ to deal with $K(x,L)$:
    \begin{align*}
        |K(x,L)| & \leq \sqrt{2} \max\left(\sqrt{L-x},\frac{1}{\sqrt{L-x}}\right) \left(\int_x^L K(x,y)^2 + (\partial_y K)^2(x,y) dy\right)^{1/2} \\
                 & \leq 2\max\left(\sqrt{L-x},\frac{1}{\sqrt{L-x}}\right) E(x)^{1/2}.
    \end{align*}
    Finally, we obtain that there exists $C>0$ such that for almost all $x\in (0,L)$,
    \[ E'(x) \leq C|\alpha+1| \max(1, L-x) E(x), \]
    and since $E(0)=0$ a Gronwall argument yields that $E(x) = 0$ for all $x\in [0,L]$. Hence \eqref{localization} holds.

    Lastly, reasoning as in Step~2 of the proof of Lemma~\ref{lem:square} by a density argument, we can easily check that the result holds true for arbitrary $f\in L^2((0,L)^2)$.
\end{proof}

We are now in position to prove Proposition~\ref{prop:estimates}.

\begin{proof}[Proof of Proposition~\ref{prop:estimates}]
    Let us begin to prove that there exists a unique weak solution $k^\alpha$ to \eqref{eq:kpde} in the sense of Definition~\ref{def:defkernel}, for any $\alpha \in \mathbb{R}$. Let $L>0$.

    \medskip

    \emph{Existence:} Denote by $K^\alpha(x,y):= k^\alpha(x,y) + \frac{\alpha}{2}x$ for all $(x,y)\in D_L$. Then, it holds that $K^\alpha$ is solution to
    \begin{equation}
        \label{eq:kpde2}
        \left\{
        \begin{aligned}
            \partial_{xx}^2 K^\alpha(x,y) - \partial_{yy}^2 K^\alpha(x,y) & = \alpha K^\alpha(x,y) - \frac{\alpha^2}{2}x & (x,y) \in D_L, \\
            \partial_y K^\alpha(x,0)                                      & = 0                                          & x \in (0,L),   \\
            K^\alpha(x,x)                                                 & = 0                                          & x \in (0,L),
        \end{aligned}
        \right.
    \end{equation}
    and it is equivalent to solve one problem or the other. Now, using Lemmas~\ref{lem:square} and~\ref{lem:localization}, we obtain that the restriction of the unique weak solution $K = \widetilde{K}^\alpha$ to \eqref{eq:complete} with $f(x,y) = f^\alpha(x,y)= - \frac{\alpha^2}{2}x \mathds{1}_{D_L}(x,y)$ to $D_L$ is a solution $K^\alpha$ to \eqref{eq:kpde2}. Besides, from Lemma~\ref{lem:localization}, it holds that $K = 0$ in $(0,L)^2 \setminus D_L$. In particular, it holds that $K$ satisfies the weak formulation: for almost all $x\in (0,L)$ and all $v\in H^1(0,L)$,
    \begin{align*}
         & \langle \partial_{xx} K(x,\cdot),v \rangle_{H^1(0,L)', H^1(0,L)} + \int_0^x \partial_y K(x,y) \partial_y v(y)\,dy \\
         & = \alpha \int_0^x \left(K(x,y) - \frac{\alpha}{2}x\right) v(y)\,dy.
    \end{align*}
    As a consequence, for all $w\in H^1(0,L)$, it holds that
    \begin{equation}
        \label{eq:integr-weak}
    \begin{aligned}
         & \int_0^L \langle \partial_{xx} K(x,\cdot),v \rangle_{H^1(0,L)', H^1(0,L)} w(x)\,dx + \int_0^L \int_0^x \partial_y K(x,y) \partial_y v(y)\,dy \,  w(x)\,dx \\
         & = \alpha
        \int_0^L \left(\int_0^x \left(K(x,y) - \frac{\alpha}{2}x\right) v(y)\,dy\right) w(x)\,dx.
    \end{aligned}
\end{equation}
    Using the fact that
    \begin{align*}
        \int_0^L \langle \partial_{xx} K(x,\cdot),v \rangle_{H^1(0,L)', H^1(0,L)} w(x)\,dx & = - \int_0^L \left(\int_0^L \partial_{x} K(x,y),v(y)\,dy \right) \partial_x w(x)\,dx \\
         & + w(L)\int_0^L \partial_x K(L,y)v(y)\,dy                                             \\
         & - w(0)\int_0^L \partial_x K(0,y)v(y)\,dy,                                            \\
    \end{align*}
    together with the fact that
    $$
        \partial_x K(0,y) = 0
    $$
    we obtain that
    \begin{align*}
        \int_0^L \langle \partial_{xx} K(x,\cdot),v \rangle_{H^1(0,L)', H^1(0,L)} w(x)\,dx & = - \int_0^L \left(\int_0^L \partial_{x} K(x,y)v(y)\,dy \right) \partial_x w(x)\,dx \\
        & + w(L)\int_0^L \partial_x K(L,y)v(y)\,dy,                                           \\
        & = - \int_0^L \left(\int_0^x \partial_{x} K(x,y)v(y)\,dy \right) \partial_x w(x)\,dx \\
        & + w(L)\int_0^L \partial_x K(L,y)v(y)\,dy.                                           \\
    \end{align*}
Combining the previous equality with \eqref{eq:integr-weak} gives
    \begin{align*}
        & - \int_0^L \left(\int_0^x \partial_{x} K v(y)\,dy \right) \partial_x w(x)\,dx  + w(L)\int_0^L \partial_x K^\alpha(L,y)v(y)\,dy \\
        & + \int_0^L \int_0^x \partial_y K(x,y) \partial_y v(y)\,dy  w(x)\,dx                                                                \\
        & = \alpha
        \int_0^L \left(\int_0^x \left(K(x,y) - \frac{\alpha}{2}x\right) v(y)\,dy\right) w(x)\, dx.                                                                                                                    
   \end{align*}

    Finally, since $k^\alpha= K - \frac{\alpha}{2}x$ in $D_L$, we obtain that $k^\alpha$ is solution to the following weak formulation: for all $v,w\in H^1(0,L)$,
    \begin{align*}
         & - \int_0^L \left(\int_0^x \partial_{x} k^\alpha(x,y)v(y)\,dy \right) \partial_x w(x)\,dx  + w(L)\int_0^L \partial_x k^\alpha(L,y)v(y)\,dy \\
         & + \int_0^L \int_0^x \partial_y k^\alpha(x,y) \partial_y v(y)\,dy  w(x)\,dx                                                                \\
         & = \alpha
        \int_0^L \left(\int_0^x k^\alpha(x,y) v(y)\,dy\right) w(x)\,dx                                                                               \\
         & - \frac{\alpha}{2}\int_0^L \left(\int_0^x v(y)\,dy \right) \partial_x w(x)\,dx  + \frac{\alpha}{2} w(L)\int_0^L v(y)\,dy                  \\
         & = \alpha
        \int_0^L \left(\int_0^x k^\alpha(x,y) v(y)\,dy\right) w(x)\,dx                                                                               \\
         & + \frac{\alpha}{2}\int_0^L  v(x) w(x)\,dx ,                                                                                            
    \end{align*}

where the last equality follows from integration by parts in the $x$ variable. We thus obtain the existence of a weak solution to \eqref{eq:kpde} in the sense of Definition~\ref{def:defkernel}.

    \medskip

    \emph{Uniqueness:} Let us now prove the uniqueness of the solution for this problem. Assume there exist two solutions $k^\alpha_1$ and $k^\alpha_2$ and denote by $\hat{k}:= k^\alpha_1 - k^\alpha_2$ their difference. Then $\hat{k}$ satisfies the homogeneous equation associated to \eqref{eq:kpde}. Since $\hat{k}$ has null trace on the diagonal $x=y$, it can be extended by $0$ to the square $(0,L)^2$. But then one can check that it satisfies the assumptions of Lemma~\ref{lem:square} without source term. Hence $\hat{k}=0$. Uniqueness is proved.

    \medskip

    \emph{Estimates:} Furthermore, Lemma~\ref{lem:square} yields the following estimate for almost all $x\in (0,L)$
    \begin{equation}
        \label{eq:H1-estimate2}
        \int_0^L \left(|\widetilde{K}^\alpha(x,y)|^2 + |\nabla \widetilde{K}^\alpha(x,y)|^2\right) dy \leq (1 + L^2)e^{C\max([\alpha]_+^{1/2}, 1)L} \|f^\alpha\|_{L^2((0,L)^2)}^2.
    \end{equation}
    This yields that for almost all $x\in (0,L)$,
    \begin{equation}
        \label{eq:H1-estimate3}
        \int_0^x \left(|K^\alpha(x,y)|^2 + |\nabla K^\alpha(x,y)|^2\right) dy \leq (1 + L^2)e^{C\max([\alpha]_+^{1/2}, 1)L} \|f^\alpha\|_{L^2((0,L)^2)}^2,
    \end{equation}
    Since $\|f^\alpha\|_{L^2((0,L)^2)}^2 \leq (\alpha L)^4$, and $K^\alpha(x,y) = k^\alpha(x,y) - \frac{\alpha}{2}x$, we obtain that
    \begin{align*}
        \int_0^x \left(|k^\alpha(x,y)|^2 + |\nabla k^\alpha(x,y)|^2\right) dy & \leq 2\int_0^x \left(|K^\alpha(x,y)|^2 + |\nabla K^\alpha(x,y)|^2\right) dy + 2\left[ \frac{\alpha^2}{12}x^3 + x \frac{\alpha^2}{4}\right], \\
                                                                              & \leq C_0\left( \alpha^2 (L^3 +L) + \alpha^4 L^4 (1 + L^2)e^{C\max([\alpha]_+^{1/2}, 1)L}\right),
    \end{align*}
    where $C_0>0$ is a constant independent of $\alpha$ and $L$.

    \medskip

    We are now in a position to conclude the proof of Proposition~\ref{prop:estimates}. For any $\lambda,\sigma >0$, we define $k_\lambda^\sigma$ a weak solution to \eqref{eq:kernel_eq} as follows: for all $t\geq 0$, $k_\lambda^\sigma|_{D_t}$ is defined as $k^\alpha$ with $L = \overline{e}(t)$ and $\alpha = \frac{\lambda}{\sigma}$. One can easily check from the previous results that $k_\lambda^\sigma$ is thus well-defined and unique and is a weak solution to \eqref{eq:kernel_eq}. $l_\lambda^\sigma$ is defined from $k_\lambda^\sigma$ according to \eqref{eq:kernel-equiv}. To get the desired estimates, it is now sufficient to apply the previously obtained estimates with $L = \be(t)$ and $\alpha = \frac{\lambda}{\sigma}>0$ for $k^\sigma_\lambda$ and $\alpha = - \frac{\lambda}{\sigma}<0$ for $l^\sigma_\lambda$. To this aim, we consider $\lambda \geq \lambda_\sigma := \sigma$. Taking into account the fact that $\be(t) \geq \be_0$ for all $t\geq 0$ then yields the existence of constants $c,C>0$ independent of $t$, $\lambda$ and $\sigma$ such that
    $$
        \int_0^x \left(|k^\sigma_\lambda(x,y)|^2 + |\nabla k^\sigma_\lambda(x,y)|^2\right) dy
        \leq Ce^{c\be(t)\sqrt{\lambda/\sigma}},
    $$
    and
    $$
        \int_0^x \left(|l^\sigma_\lambda(x,y)|^2 + |\nabla l^\sigma_\lambda(x,y)|^2\right) dy
        \leq C\left( \frac{\lambda}{\sigma}\right)^4e^{c\be(t)}.
    $$
    Hence, \eqref{k-estimate} and \eqref{l-estimate} hold.

\end{proof}

\subsection{Proofs of Lemmas~\ref{lem:map} and \ref{lem:back-map}}
\label{proof:lems}
We begin with the following lemma, from which we will easily deduce Lemmas~\ref{lem:map} and \ref{lem:back-map}. The sets $D^{\rm targ}$ and $D^{\rm ini}$ are respectively defined before Definitions \ref{def:weak-target} and \ref{def:weak-solution3}.

\begin{lemma}
    \label{lem:duality}
    Let $\sigma, \lambda, \tau_1$ and some initial conditions $\zeta_\lambda^{\sigma,\tau_1},g_\lambda^{\sigma,\tau_1}$ be defined as in Lemmas~\ref{lem:map} and \ref{lem:back-map}. Assume that some functions $\zeta,g \in \left[L^2(0,T;H^1) \right]_{\be}$ such that $\partial_t \zeta, \partial_t g \in \left[ L^2(0,T;(H^1)') \right]_{\be}$ are related to each other by the relation: for any $t \geq \tau_1$, $g(t) = \mathcal{T}_{\lambda,t}^\sigma \zeta(t)$ (or, equivalently, from Lemma~\ref{lem:inv-map}, $\zeta(t) =  \mathcal{T}_{\lambda,t}^{\sigma,\rm inv} g(t)$).
    Then the following assertions hold:
    \begin{itemize}
        \item[i)] The linear operator $\mathcal{G}: \left[ L^2((0,T), L^2)\right]_{\overline{e}} \to \left[ L^2((0,T), L^2)\right]_{\overline{e}}$ defined  for any $f \in \left[ L^2((0,T), L^2)\right]_{\overline{e}}$ by
              \begin{equation}
                  \label{eq:adjoint}
                  \mathcal{G} f(t,y) = f(t,y) - \int_y^{\be(t)} k_\lambda^\sigma(x,y) f(t,x) \,dx, \quad \mbox{ for a.a. } t\in (0,T), \; x\in (0, \be(t)) ,
              \end{equation}
              is invertible from $D^{\rm targ}$ to $D^{\rm ini}$.
        \item[ii)] For any test function $v \in D^{\rm targ}$, it holds
              \begin{equation}
                  \label{eq:duality-1}
                  a^{\rm targ}(g,v) = a^{\rm ini}(\zeta,\mathcal{G}v),
              \end{equation}
              where $a^{\rm ini}$ and $a^{\rm targ}$ are given respectively by \eqref{eq:defaini} and \eqref{eq:defatar}.
        \item[iii)] As a consequence of $i)$ and $ii)$, for any test function $w \in D^{\rm ini}$, it holds
              \begin{equation}
                  \label{eq:duality-2}
                  a^{\rm targ}(g,\mathcal{G}^{-1} w) = a^{\rm ini}(\zeta,w).
              \end{equation}
    \end{itemize}
\end{lemma}

\begin{proof}
    Let $\lambda, \sigma >0$. To simplify, we denote in the sequel $k:= k_\lambda^\sigma$.
    \begin{itemize}

        \item[i)]  Let $v \in D^{\rm targ}$ and differentiate \eqref{eq:adjoint}. It holds for almost any $t \geq 0, y \in (0,\be(T))$,
              \begin{equation}
                  \partial_y(\mathcal{G} v)(t,y) = \partial_y v(t,y) + k(y,y)v(t,y) - \int_y^{\be(t)} \partial_y k(x,y) v(t,x) dx
              \end{equation}
              The previous equality holds in $\left[L^2(0,T;L^2)\right]_{\be}$ since
              $$
                  v \in \left[L^2(0,T;H^2) \right]_{\be} \subset \left[L^2(0,T;L^\infty) \right]_{\be},
              $$
              and the function $(0,\be(t)) \ni y \mapsto k(y,y)$ belongs to $H^1(0,\be(t)) $. Besides, the quantity  $\left\|\partial_y k(x,\cdot)\right\|_{L^2(0,x)}$ is bounded uniformly in $x$ for $x\in (0,\be(t))$.
              Differentiate once again:
             \begin{align*}
                  \partial_{yy}^2 (\mathcal{G} v)(t,y) = & \partial_{yy}^2 v(t,y) + \left(\frac{d}{dy} k(y,y) \right) v(t,y) + k(y,y) \partial_y v(t,y) + \partial_y k(y,y) v(t,y) \\
                                                         & - \langle \partial_{yy}^2 k(\cdot,y), v(t) \rangle_{(H^1(y,\overline{e}(t)))',H^1(y,\overline{e}(t))}.
              \end{align*}
              All the terms on the right-hand-side belong to $\left[L^2(0,T;L^2)\right]_{\be}$. Therefore $\mathcal{G} v \in \left[L^2(0,T;H^2) \right]_{\be}$. It is then clear that $$\mathcal G v \in D^{\rm ini},$$ since in particular
              \[ \sigma \partial_y (\mathcal{G} v)(t,\be(t)) = 0 + \sigma k(\be(t),\be(t))v(t,\be(t)) - 0 = K_l(t)(\mathcal G v)(t,e(t)). \]

              Therefore the range of $\mathcal{G}$ is a subset of $D^{\rm ini}$. Besides, $\mathcal{G}$ is invertible in $L^2$ from classical results on Volterra operators (see Lemma~\ref{lem:inv-map}). Finally, just as in Lemma~\ref{lem:inv-map}, it can be easily checked following the same lines that the inverse has a similar form, and that it is defined from $D^{\rm ini}$ with values in $D^{\rm targ}$.

        \item[ii)]Let $v \in D^{\rm targ}$. It holds, denoting by $\phi(t,x):= \int_0^x k(x,y)\zeta(t,y)\,dy$,
              \begin{equation}
                  \begin{split}
                      \label{eq:lem3intpart}
                      &a^{\rm targ}(g,v) = \int_{\tau_1}^{T}
                      \left\langle \partial_t v(t) + \sigma \partial_{xx}^2 v(t) - \lambda v(t), g(t)\right\rangle_{H^1(0,\be(t))', H^1(0,\be(t)} dt+ \int_0^{\be(\tau_1)}g(\tau_1,x) v(\tau_1,x) dx \\
                      &=\int_{\tau_1}^{T} \left\langle  \partial_t v(t) + \sigma \partial_{xx}^2 v(t) - \lambda v(t), \zeta(t)\right\rangle_{H^1(0,\be(t))', H^1(0,\be(t))} dt\\
                      & -\int_{\tau_1}^{T} \left\langle\partial_t v(t) + \sigma \partial_{xx}^2 v(t) - \lambda v(t),  \phi(t) \right\rangle_{H^1(0,\be(t))', H^1(0,\be(t))}
                      dt\\
                      +&\int_0^{\be(\tau_1)}\zeta(\tau_1,x) v(\tau_1,x) dx\\
                      &-\int_0^{\be(\tau_1)}\left(\int_{0}^{x}k(x,y)\zeta(\tau_1,y) dy\right)v(\tau_1,x) dx
                  \end{split}
              \end{equation}

              Let us now look at the term in \eqref{eq:lem3intpart} involving the function $\phi$ and perform some integration by parts. Begin with the time derivative: it holds,
              \begin{align*}
                   & \int_{\tau_1}^T
                  \left\langle \partial_t v(t), \phi(t) \right\rangle_{H^1(0,\be(t))', H^1(0,\be(t))} dt                                       \\
                   & = - \int_{\tau_1}^T
                  \left\langle \partial_t \phi(t), v(t) \right\rangle_{H^1(0,\be(t))', H^1(0,\be(t))} dt + \int_0^ {\overline{e}(\tau_1)} v(\tau_1,x)\phi(\tau_1,x)\,dx
                  \\
                   & - \overline{v} \int_{\tau_1}^T
                  \phi(t,\be(t))v(t,\be(t))                                                                                                    \\
                   & = - \int_{\tau_1}^T \int_0^{\be(t)}  \langle \partial_t \zeta(t), k(x,\cdot)\rangle_{H^{1}(0,x)', H^1(0,x)}  v(t,x) dx dt \\
                   & - \int_0^{\be(\tau_1)} \left(\int_0^x k(x,y)\zeta(\tau_1,y)  dy \right) v(\tau_1,x) dx                                    \\
                   & - \overline{v} \int_{\tau_1}^T
                  \phi(t,\be(t))v(t,\be(t)).                                                                                                   \\
              \end{align*}

              Now the space derivative:
              \begin{align*}
                   & \int_{\tau_1}^T
                  \left\langle \sigma \partial_{xx}^2 v(t), \phi(t)\right\rangle_{H^1(0,\be(t))', H^1(0,\be(t))} dt                                    \\
                   & = - \int_{\tau_1}^T
                  \int_0^{\be(t)} \left( k(x,x) \zeta(t,x) + \int_0^x \partial_x k(x,y) \zeta(t,y) dy \right) \sigma \partial_x v(t,x)dx dt            \\
                   & = \sigma \int_{\tau_1}^T \int_0^{\be(t)} \left[ \frac{d}{dx} k(x,x) \zeta(t,x) + k(x,x)\partial_x \zeta(t,x) \right] v(t,x) dx dt \\
                   & - \int_{\tau_1}^T v(t,\be(t)) \left[ \sigma k(\be(t),\be(t))\zeta(t,\be(t)) \right] dt                                            \\
                   & - \sigma \int_{\tau_1}^T  \int_0^{\be(t)} \left( \int_0^x \partial_x k(x,y) \zeta(t,y) dy \right)  \partial_x v(t,x)dx dt.        \\
              \end{align*}
              Using the weak formulation of $k$, we obtain that for all $t\geq \tau_1$
              \begin{align*}
                   & - \sigma  \int_0^{\be(t)} \left( \int_0^x \partial_x k(x,y) \zeta(t,y) dy \right)  \partial_x v(t,x)dx dt \\
                   & = - v(t, \overline{e}(t)) \int_0^{\be(t)} \sigma  \partial_x k(\be(t), y)\zeta(t,y)\,dy                   \\
                   & - \sigma \int_0^{\be(t)} v(t,x)
                  \left(\int_0^x \partial_y k(x,y) \partial_y \zeta(t,y)\,dy\right) \,dx                                       \\
                   & + \lambda \int_0^{\be(t)} \int_0^x k(x,y) \zeta(t,y)\,dy v(t,x)\,dx                                       \\
                   & - \frac{\lambda}{2} \int_0^{\be(t)}  \zeta(t,x)\ v(t,x)\,dx.
              \end{align*}
              Remember that:
              \[  H_{nl}(t)\zeta(t) = \int_0^{\be(t)} \left[\sigma \partial_x k(\be(t),y) + \bv k(\be(t),y)\right] \zeta(t,y) dy, \]
              and now insert the two previous calculations into \eqref{eq:lem3intpart}. It holds that:
              \footnotesize
              \begin{align*}
                  a^{\rm targ}(g,v) & = \int_{\tau_1}^T \left\langle \partial_t v(t) + \sigma \partial_{xx}^2 v(t) - \lambda v(t), \zeta(t) \right\rangle_{H^1(0,\be(t))', H^1(0,\be(t))} + \int_{\tau_1}^T H_{nl}(t) \zeta(t) v(t,\be(t)) dt \\
                                    & + \int_{\tau_1}^T \int_0^{\be(t)} \langle \partial_t \zeta(t), k(x,\cdot)\rangle_{(H^{1}(0,x))', H^1(0,x)} v(t,x) dx dt                                                                                 \\
                                    & - \sigma \int_{\tau_1}^T \int_0^{\be(t)} \left[ \frac{d}{dx} k(x,x) \zeta(t,x) + k(x,x)\partial_x \zeta(t,x)\right] v(t,x) dx dt                                                                        \\
                                    & + \sigma \int_{\tau_1}^T \int_0^{\be(t)}
                  \left(\int_0^x \partial_y k(x,y) \partial_y \zeta(t,y)\,dy\right) v(t,x)\,dx \, dt                                                                                                                                          \\
                                    & +\frac{\lambda}{2}\int_{\tau_1}^T  \int_0^{\be(t)}  \zeta(t,x) v(t,x)\,dx                                                                                                                               \\
                                    & + \int_{\tau_1}^T v(t,\be(t)) K_l(t)\zeta(t,\be(t)) dt + \int_0^{\be(\tau_1)}\zeta(\tau_1,x) v(\tau_1,x) dx.                                                                                            \\
              \end{align*}
              \normalsize
              Note that we have
              $$
                  \frac{\lambda}{2}\int_0^{\be(t)} \zeta(t,x) v(t,x)\,dx
                  = -\sigma \int_{0}^{\be(t)} \frac{d}{dx}k(x,x) \zeta(t,x) v(t,x)\,dx.
              $$
              Hence, we obtain that
              \begin{align*}
                  a^{\rm targ}(g,v) & = \int_{\tau_1}^T \left\langle \partial_t v(t) + \sigma \partial_{xx}^2 v(t) , \zeta(t) \right\rangle_{H^1(0,\be(t))', H^1(0,\be(t))} + \int_{\tau_1}^T H_{nl}(t) \zeta(t) v(t,\be(t)) dt \\
                                    & + \int_{\tau_1}^T \int_0^{\be(t)} \langle \partial_t \zeta(t), k(x,\cdot)\rangle_{(H^{1}(0,x))', H^1(0,x)} v(t,x) dx dt                                                                   \\
                                    & - \sigma \int_{\tau_1}^T \int_0^{\be(t)}  k(x,x)\partial_x \zeta(t,x) v(t,x) dx dt                                                                                                        \\
                                    & + \sigma \int_{\tau_1}^T \int_0^{\be(t)}
                  \left(\int_0^x \partial_y k(x,y) \partial_y \zeta(t,y)\,dy\right) v(t,x)\,dx \, dt                                                                                                                            \\
                                    & + \int_{\tau_1}^T v(t,\be(t)) K_l(t)\zeta(t,\be(t)) dt + \int_0^{\be(\tau_1)}\zeta(\tau_1,x) v(\tau_1,x) dx.                                                                              \\
              \end{align*}

              Let us now denote by $w:= \mathcal{G} v$ and by $\psi(t,y):= \int_y^{\overline{e}(t)} k(x,y) v(t,x)\,dx$. It then holds that
              \begin{align*}
                  a^{\rm ini}(\zeta, w) & := \int_{\tau_1}^T
                  \left\langle \partial_t w(t) + \sigma \partial_{xx}^2 w(t),  \zeta(t)
                  \right\rangle_{H^1(0,\overline{e}(t))',H^1(0,\overline{e}(t))}  dt + \int_{\tau_1}^T H_{nl}(t) \zeta(t) w(t,\be(t)) dt      \\
                                        & +\int_0^{\be(\tau_1)}\zeta(\tau_1,x) w(\tau_1,x) dx,                                                \\
                                        & = \int_{\tau_1}^T
                  \left\langle \partial_t v(t) + \sigma \partial_{xx}^2 v(t),  \zeta(t)
                  \right\rangle_{H^1(0,\overline{e}(t))',H^1(0,\overline{e}(t))}  dt + \int_{\tau_1}^T H_{nl}(t) \zeta(t) v(t,\be(t)) dt      \\
                                        & +\int_0^{\be(\tau_1)}\zeta(\tau_1,x) v(\tau_1,x) dx                                                 \\
                                        & - \int_{\tau_1}^T
                  \left\langle \partial_t \psi(t) + \sigma \partial_{xx}^2 \psi(t),  \zeta(t)
                  \right\rangle_{H^1(0,\overline{e}(t))',H^1(0,\overline{e}(t))}  dt  - \int_0^{\be(\tau_1)}\zeta(\tau_1,x) \psi(\tau_1,x) dx \\
              \end{align*}
              Doing similar computations as above, we obtain that
              \begin{align*}
                  \int_{\tau_1}^T \langle \partial_t \psi(t), \zeta(t)\rangle_{H^1(0, \be(t))', H^1(0,\be(t))} & = - \int_{\tau_1}^T   \langle \partial_t \zeta(t), \psi(t)\rangle_{H^1(0, \be(t))', H^1(0,\be(t))}                        \\
                                                                                                               & - \int_0^{\be(\tau_1)}\zeta(\tau_1,y)\psi(\tau_1,y)\,dy
                  - \overline{v} \int_{\tau_1}^T \zeta(t,\be(t))\psi(t, \be(t))\,dt                                                                                                                                                        \\
                                                                                                               & = - \int_{\tau_1}^T \int_0^{\be(t)} \langle \partial_t \zeta(t), k(x,\cdot)\rangle_{(H^{1}(0,x))', H^1(0,x)} v(t,x) dx dt \\
                                                                                                               & - \int_0^{\be(\tau_1)}\zeta(\tau_1,y)\psi(\tau_1,y)\,dy - \overline{v} \int_{\tau_1}^T \zeta(t,\be(t))\psi(t, \be(t))\,dt \\
                                                                                                               & = - \int_{\tau_1}^T \int_0^{\be(t)} \langle \partial_t \zeta(t), k(x,\cdot)\rangle_{(H^{1}(0,x))', H^1(0,x)} v(t,x) dx dt \\
                                                                                                               & - \int_0^{\be(\tau_1)}\zeta(\tau_1,y)\psi(\tau_1,y)\,dy.                                                                  \\
              \end{align*}

              Moreover, since $\partial_y \psi (t,y) = k(y,y) v(t,y) - \int_y^{\overline{e}(t)} \partial_y k(x,y)v(t,x)\,dx$, we have
              \begin{align*}
                   & \int_{\tau_1}^T \langle \sigma\partial_{xx} \psi(t), \zeta(t)\rangle_{H^1(0, \be(t))', H^1(0,\be(t))}                                                                                                        \\
                   & = - \sigma\int_{\tau_1}^T \int_0^{\be(t)} \partial_y \zeta(t,y) \partial_y \psi(t,y)\,dy + \sigma\int_{\tau_1}^T \zeta(t, \be(t)) k(\be(t), \be(t)) v(t, \be(t)),                                            \\
                   & = - \sigma\int_{\tau_1}^T \int_0^{\be(t)} \partial_y \zeta(t,y) k(y,y) v(t,y)\,dy + \sigma \int_{\tau_1}^T \int_0^{\be(t)} \partial_y \zeta(t,y) \left(\int_y^{\be(t)}\partial_y k(x,y)v(t,x)\,dx\right)\,dy \\
                   & + \sigma\int_{\tau_1}^T \zeta(t, \be(t)) k(\be(t), \be(t)) v(t, \be(t)),                                                                                                                                     \\
                   & = - \sigma\int_{\tau_1}^T \int_0^{\be(t)} \partial_y \zeta(t,y) k(y,y) v(t,y)\,dy + \sigma \int_{\tau_1}^T \int_0^{\be(t)} v(t,x)
                  \left(\int_0^{x}\partial_y \zeta(t,y) \partial_y k(x,y)\,dy\right) \,dx                                                                                                                                         \\
                   & + \int_{\tau_1}^T K_l(t)\zeta(t) v(t, \be(t)).                                                                                                                                                               \\
              \end{align*}
              As a consequence, we obtain that
              $$
                  a^{\rm targ}(g,v) =  a^{\rm ini}(\zeta,w).
              $$

              Hence the desired result.

        \item[iii)] The proof of (iii) is a direct consequence of (i) and (ii).
    \end{itemize}
\end{proof}

Now we provide the proofs of Lemmas~\ref{lem:map} and \ref{lem:back-map}.

\begin{proof}[Proof of Lemma~\ref{lem:map} and \ref{lem:back-map}]
    Let $\zeta_\lambda^\sigma$ be a weak-$L^2$ solution to \eqref{linearized_scal} in the sense of Definition~\ref{def:weak-solution3}. Define now, for all $t\geq \tau_1$, $x \in (0,\be(t))$,
    $$
        g_\lambda^\sigma(t,x):= \mathcal T_{\lambda,t}^\sigma \zeta_\lambda^\sigma(t,x) = \zeta_\lambda^\sigma(t,x) -\int_0^x k_\lambda^\sigma(x,y)\zeta_\lambda^\sigma(t,y)\,dy.
    $$
    \emph{Continuity and initial data: }
    $\zeta_\lambda^\sigma \in \left[\cC^0([\tau_1,T]);L^2)\right]_{\be}$ by assumption and the Cauchy-Schwarz inequality provides the following estimate, for any $\tau_1 \leq s,t \leq T$,
    \begin{equation*}
        \left\| \int_0^x k_\lambda^\sigma(x,y) \left(\zeta_\lambda ^\sigma(t,y)-\zeta_\lambda^\sigma(s,y) \right) dy \right\|_{L^2(0,\be(T))} \leq \|k_\lambda^\sigma\|_{L^2(D_T)} \| \|\zeta_\lambda^\sigma(t)-\zeta_\lambda^\sigma(s) \|_{L^2((0,\be(T)))},
    \end{equation*}
    which goes to 0 by assumption as $t$ goes to $s$. Therefore $g_\lambda^\sigma \in \left[\cC^0([\tau_1,T]),L^2)\right]_{\be}$ as well. The initial data $g_{\lambda}^{\sigma,\tau_1} = \mathcal{T}_{\lambda,\tau_1}^{\sigma} \zeta_{\lambda}^{\sigma,\tau_1}$ follows from continuity and the initial data of $\zeta_{\lambda}^\sigma$.

    \medskip

    \emph{Time derivative :} We want to differentiate this formula with respect to time. It gives formally
    for almost any $t \in (\tau_1,T)$, $x \in (0,\be(t))$
    \[ \partial_t g_\lambda^\sigma(t,x) = \partial_t \zeta_\lambda^\sigma(t,x) - \int_0^x k_\lambda^\sigma(x,y) \partial_t \zeta_\lambda^\sigma(t,y) dy. \]
    By assumption, $\partial_t \zeta_\lambda^\sigma \in \left[L^2((\tau_1,T);(H^1)' \right]_\be$ and from Proposition~\ref{prop:estimates} it holds that for almost all $x \in (0,\be(T))$, $k_\lambda^\sigma(x;\cdot) \in H^1(0,x)$, uniformly in $x$. Therefore the integral terms are well-defined as duality products:
    \[  \partial_t g_\lambda^\sigma(t,x) = \partial_t \zeta_\lambda^\sigma(t,x) - \langle \partial_t \zeta_\lambda^\sigma(t), k_\lambda^\sigma(x) \rangle_{(H^1(0,x))',H^1(0,x)},  \]
    where the second term can be estimated as:
    \[\langle \partial_t \zeta_\lambda^\sigma(t), k_\lambda^\sigma(x) \rangle_{(H^1(0,x))',H^1(0,x)} \leq \| \partial_t \zeta_\lambda^\sigma(t) \|_{(H^1(0,x))'} \|k_\lambda^\sigma(x)\|_{H^1(0,x)} \leq \|\partial_t \zeta_\lambda^\sigma(t) \|_{(H^1(0,\be(t)))'} \sup_{0 \leq x \leq \be(T)} \|k_\lambda^\sigma(x)\|_{H^1(0,x)},    \]
    where the last term is independent of $x$ and belongs to $L^2(\tau_1,T)$ by assumption. Hence $\partial_t g_\lambda^\sigma \in  \left[L^2((\tau_1,T);(H^1)')\right]_{\be}.$

    \medskip

    Therefore, $g_\lambda^\sigma$ and $\zeta_\lambda^\sigma$ satisfy the assumptions of Lemma~\ref{lem:duality}. It follows that for any $v \in D^{\rm targ}$, $a^{\rm targ}(g_\lambda^\sigma,v) = 0$, so that $g_\lambda^\sigma$ satisfies indeed Definition~\ref{def:weak-target}.

    \medskip

    The proof of Lemma~\ref{lem:back-map} follows the exact same lines.

\end{proof}

\subsection{Proof of Theorem~\ref{thm:1}} \label{proof:thm1}

Fix $\mu>0$ and $(\bu,\be)$. We need to check all the conditions of Definitions~\ref{def:weak-solution}-\ref{def:exp-stab}.

\medskip

Let us first deal with the thickness and remember from \eqref{thickness}-\eqref{eq:thickness-variable} that $\delta e'(t) = \delta \theta(t)$.
The exponential stabilization of $\delta e$ can then be achieved with no effort. It suffices to define $\Theta(t,w) = -\mu w$ for all $t\geq0$ and $w\in \mathbb{R}$ to get \eqref{eq:def-stability-thickness} with $C_{\bar{e},\mu} = 1$.

\medskip

Let us now focus on the exponential stabilization of $\delta u$ with the control variables $\delta \psi$ in \eqref{linearized}. Remember the decomposition \eqref{linearized_decomp} and choose $\lambda>0$ such that $\lambda \geq \max_{1\leq i \leq n}\lambda_{\sigma_i}$ where $\lambda_{\sigma_i}$ is defined as in Corollary~\ref{cor:exp-stability}.

Then it follows by Proposition~\ref{prop:estimates} that there exists a unique solution $k_\lambda^{\sigma_i}$ (respectively $l_\lambda^{\sigma_i}$) to the kernel problem \eqref{eq:kernel_eq} (respectively to the inverse kernel problem \eqref{eq:kernel_inv_eq}) satisfying estimates \eqref{eq:estimate-k} and \eqref{eq:estimate-l} with $\sigma = \sigma_i$. Then, for all $1\leq i \leq n$, $t\geq 0$ and $z\in H^1(0, \be(t))$, let us define
\begin{align}
    \label{eq:Hltot}
    H_{l,\lambda}^{\sigma_i}(t)z   & := \sigma_i k^{\sigma_i}_\lambda(\be(t),\be(t)) z(\be(t)),                                                                      \\
    \label{eq:Hnl:tot}
    H_{nl, \lambda}^{\sigma_i}(t)z & = \int_0^{\be(t)} \left[\sigma_i \partial_x k^{\sigma_i}_\lambda(\be(t),y) + \bv k^{\sigma_i}_\lambda(\be(t),y)\right] z(y) dy,
\end{align}
Defining now
\[ \delta \psi^{\sigma_i, \lambda}(t) := H_{l,\lambda}^{\sigma_i}(t)\zeta_\lambda^{\sigma_i}(t) + H_{nl,\lambda}^{\sigma_i}(t)\zeta_\lambda^{\sigma_i}(t), \]
it follows from Corollary~\ref{cor:exp-stability} that there exists a unique weak-$L^2$ solution $\zeta_{\lambda}^{\sigma_i}$ to \eqref{linearized_scal} in the sense of Definition~\ref{def:weak-solution3}) with $\tau_1 =0$ and $\zeta_\lambda^{\sigma_i,0} = z_i^0$. To simplify notations, we will denote by $z^\lambda_i$ the solution $\zeta_{\lambda}^{\sigma_i}$. Note that $z^\lambda_i$ is then also solution to problem \eqref{linearized_decomp} with $\delta \psi^i$ given by
\[\delta \psi^i(t) =  \delta \psi^{\sigma_i, \lambda}(t) := H_{l,\lambda}^{\sigma_i}(t)z_i(t) + H_{nl,\lambda}^{\sigma_i}(t)z_i(t).\]

From Corollary~\ref{cor:exp-stability}, there exist constants $C,c>0$ independent of $\lambda$, $1\leq i \leq n$ and $t$ such that for any $1 \leq i \leq n$, and $t \geq 0$,
\[ \|z^\lambda_i(t)\|_{L^2(0,\be(t))} \leq C e^{c \be_0 \sqrt{\lambda/\sigma_i} + c \be(t) - \lambda t} \|z_i^0\|_{L^2(0,\be_0)}.  \]
It follows that there exists a $\lambda_\mu>0$ large enough and a constant $C_\mu >0$ that depends on $(\max_{1 \leq i \leq n} \sigma_i, \lambda_\mu,\be_0)$ such that for any $t \geq 0$,
\[\|z^{\lambda_\mu}(t)\|_{L^2(0,\be(t))^n} \leq C_\mu e^{- \mu t} \|z^0\|_{L^2(0,\be_0)^n} , \]
where $z^{\lambda_\mu}:=(z_i^{\lambda_\mu})_{1\leq i \leq n}$ and $z^0:=(z_i^0)_{1\leq i \leq n}$.
It remains to check that the operators $(H_{l,\lambda_\mu}(t))_{t\geq 0}$ and $(H_{nl, \lambda_\mu}(t))_{t\geq 0}$ defined such that, for any $t\geq 0$ and $z:=(z_i)_{1\leq i \leq n}\in H^1(0, \overline{e}(t))^n$,
\begin{align*}
    H_{l,\lambda_\mu}(t)z   & := Q(\overline{u})^{-1} \left( H^{\sigma_i}_{l,\lambda_\mu}(t)z_i\right)_{1\leq i \leq n},  \\
    H_{nl, \lambda_\mu(}t)z & := Q(\overline{u})^{-1} \left( H^{\sigma_i}_{nl,\lambda_\mu}(t)z_i\right)_{1\leq i \leq n},
\end{align*}
satisfy assumptions (P1)-(P2)-(P3). To prove this, it is sufficient to show that for all $1\leq i \leq n$, the families of operators $(H_{l,\lambda_\mu}^{\sigma_i}(t))_{t\geq 0}$ and $(H_{nl, \lambda_\mu}^{\sigma_i}(t))_{t\geq 0}$ satisfy the scalar assumptions (P1')-(P2')-(P3').

Let $1 \leq i \leq n$.  It follows from Proposition~\ref{prop:estimates} that for all $t \geq 0$, the functions $(0,\be(t)) \ni y \to \partial_x k_{\lambda_{\mu}}^{\sigma_i}(t,\be(t),y)$ and $(0, \be(t))\ni y \to k_{\lambda_{\mu}}^{\sigma_i}(t,\be(t),y)$ are well-defined almost everywhere and belong to $L^2((0,\be(t))$. As a consequence, $H_{nl,\lambda_\mu}^{\sigma_i}(t)$ defined by \eqref{eq:Hnl:tot} is well-defined and satisfies (P1').  Now in order to check (P2'), for any $T>0$, it holds by the Cauchy-Schwarz inequality that for all $z \in [L^2((0,T), L^2]_{\be}$,
\[ \|H_{nl,\lambda_\mu}^{\sigma_i}(\cdot) z(\cdot)\|_{L^2(0,T)}^2 \leq \int_0^T \left(\int_0^{\be(t)} \left[\sigma \partial_x k_{\lambda_{\mu}}^{\sigma_i}(\be(t),y) + V k_{\lambda_{\mu}}^{\sigma_i}(\be(t),y)\right]^2 dy \right) \|z(t)\|_{L^2(0,\be(t))}^2 dt.  \]
But according to the uniform estimate \eqref{k-estimate} in Proposition~\ref{prop:estimates}, it holds that for any $0 \leq t \leq T$,
\[ \int_0^{\be(t)} \left[\sigma_i \partial_x k_{\lambda_\mu}^{\sigma_i}(\be(t),y) + \bv k_{\lambda_\mu}^{\sigma_i}(\be(t),y)\right]^2 \leq C(\sigma_i,\bv) e^{c\be(T) \sqrt{\lambda_{\mu}/\sigma_i}},  \]
hence (P2').

Finally, the family $(H_{l,\lambda_\mu}^{\sigma_i}(t))_{t\geq 0}$ satisfies (P3'), provided that $\mathbb{R}_+^* \ni t \to k_{\lambda_{\mu}}^{\sigma_i}(t,\be(t),\be(t))$ belongs to $L^{\infty}_{loc}\left(\mathbb{R}_+^*; \mathbb{R}\right)$. It is in fact again a consequence of Proposition~\ref{prop:estimates} and the one-dimensional Sobolev inequality.
Indeed, for $t \geq 0$, we then write for all $(y,\tilde{y}) \in [0,\be(t)]^2$,

\[ k_{\lambda_{\mu}}^{\sigma_i}(\be(t),y) = \int_{\tilde{y}}^y \partial_y k_{\lambda_{\mu}}^{\sigma_i}(\be(t),y') \,dy' + k_{\lambda_{\mu}}^{\sigma_i}(\be(t),\tilde{y}).\]

Integration with respect to $\tilde{y} \in [0,\be(t)]$ leads to:
\begin{equation*}
    \begin{aligned}
        |k_{\lambda_{\mu}}^{\sigma_i}(\be(t),y)| & \leq \frac{1}{\be(t)} \int_0^{\be(t)} \int_{\tilde{y}}^y |\partial_y k_{\lambda_{\mu}}^{\sigma_i}(\be(t),y')| \,dy' \,d\tilde{y} + \frac{1}{\be(t)}\int_0^{\be(t)} |k_{\lambda_{*}}^{\sigma_i}(\be(t),\tilde{y})| d\tilde{y} \\
                                                 & \leq C \left(\sqrt{\be(t)} + \frac{1}{\sqrt{\be(t))}} \right) e^{c\be(t) \sqrt{\lambda_{\mu}/\sigma_i}},
    \end{aligned}
\end{equation*}
where we have used again estimate \eqref{k-estimate} as well as Cauchy-Schwarz inequality. Hence (P3') and the proof of Theorem~\ref{thm:1}.

\subsection{Proof of Theorem~\ref{thm:2}}
\label{proof:thm2}
Let us fix $T>0$. Let us first explain how the function $\Theta$ can be chosen to ensure condition b) of Definition~\ref{def:finite-stab}. The main idea is to go from the autonomous feedback $\Theta(w)=-\mu w$ to \emph{piecewise constant} in time.
Let $(t'_m)_{m\in\mathbb{N}}$ be an increasing sequence of real numbers such that $t'_0 = 0$ and $\displaystyle t'_m \mathop{\longrightarrow}_{m \to \infty} T^-$. Let $(\mu_m)_{m \in \N}$ be a nondecreasing sequence of positive numbers which will be made precise below and define, for all $t\geq 0$ and $w\in \mathbb{R}$,
$$
    \Theta(t,w) = - \mu_m w \quad \mbox{ if } t\in [t'_m, t'_{m+1}).
$$
Then, it holds that for all $m \in \N$ and for all $t \in [t'_m,t'_{m+1})$,
\begin{equation}
    \label{eq:induction}
    |\delta e(t)| \leq e^{-\mu_m(t-t'_m)} |\delta e(t'_m)| \leq e^{-\sum_{k=0}^{m-1} \mu_k (t'_{k+1}-t'_k)} |\delta e(0)|
\end{equation}
It is then clear that condition b-i) is always satisfied with this choice. Moreover, if the series $\displaystyle\sum_{k \in \N} \mu_k (t'_{k+1}-t'_k)$ diverges, then $\displaystyle \delta e(t) \mathop{\longrightarrow}_{t\to T^-} 0$. For instance, this is the case when defining $t'_m = T-\frac{1}{m}$ and $\mu_m=m$ for all $m\geq1$. Hence b-ii).

\medskip

Let us now turn to the proof of condition a) of Definition~\ref{def:finite-stab}. Let us introduce again an increasing sequence $(t_m)_{m\in\mathbb{N}}$ of real numbers such that $t_0 = 0$ and $\displaystyle t_m \mathop{\longrightarrow}_{m \to \infty} T^-$. Let us introduce a sequence of positive numbers $(\lambda_m)_{m\in\mathbb{N}}$ such that $\displaystyle \lambda_m \geq \mathop{\max}_{1\leq i \leq n} \lambda_{\sigma_i}$ for all $m\in \mathbb{N}$, where $\lambda_{\sigma_i}$ is defined in Proposition~\ref{prop:estimates} for $\sigma = \sigma_i$. Then, for all $1\leq i \leq n$, we define the families of operators $(H^i_{nl}(t))_{t\geq 0}$ and $(H^i_{l}(t))_{t\geq 0}$ as follows:
$$
    H^i_{nl}(t) = H^{\sigma_i}_{nl,\lambda_m}(t) \quad \mbox{ and } H^i_{nl}(t) = H^{\sigma_i}_{l,\lambda_m}(t) \mbox{ if } t \in [t_m, t_{m+1}).
$$
We also define for all $t\geq 0$ and $z:=(z_i)_{1\leq i \leq n}\in H^1(0, \overline{e}(t))^n$,
\begin{align} \label{defHl}
    H_{l}(t)z  & := Q(\overline{u})^{-1} \left( H^{i}_{l}(t)z_i\right)_{1\leq i \leq n},  \\ \label{defHnl}
    H_{nl}(t)z & := Q(\overline{u})^{-1} \left( H^{i}_{nl}(t)z_i\right)_{1\leq i \leq n}. \\\nonumber
\end{align}
We wish to identify some sufficient conditions on the sequence $(\lambda_m)_{m\in\mathbb{N}}$ and $(t_m)_{m\in\mathbb{N}}$ in order to guarantee condition a) of Definition~\ref{def:finite-stab}.

\medskip

To this aim, we first prove the following lemma.

\begin{lemma}\label{prop:finite}
    Let $(\lambda_m)_{m\in\N}$ be a nondecreasing sequence of positive coefficients and let $(t_m)_{m\in\N}$ be an increasing sequence of times such that $t_0=0$ and $\displaystyle t_m \mathop{\longrightarrow}_{m \to \infty} T^-$. Let us define, for $m \geq 0$, $s_m:= \sum_{k=0}^{m} \lambda_k (t_{k+1}-t_k)$.
    Then, there exists a constant $\gamma>0$ such that, if
    \begin{equation}
        \label{hyp1}
        \forall m\in \mathbb{N}, \quad (t_{m+1}-t_m)\sqrt{\lambda_m} \geq \gamma,
    \end{equation}
    then, there exists positive constants $C>0$  and $\alpha >0$ such that for any $m\in\mathbb{N}$ and any $t\in [t_m,t_{m+1})$,
    \begin{equation}
        \|z(t)\|_{L^2(0,\be(t))^n} \leq C e^{-s_{m}+\alpha m} \|z_0\|_{L^2(0,\be_0)^n},
    \end{equation}
    where $z:=(z_i)_{1\leq i \leq n}$ with $z_i$ the unique weak solution of \eqref{linearized_decomp} and $\delta \psi^i$ defined by:
    $$
        \forall t\geq 0, \quad \delta \psi_i(t)  = H^i_l(t)z_i(t) + H^i_{nl}(t)z_i(t).
    $$
    Besides, if we assume in addition that:
    \begin{equation}
        \label{hyp2}
        \lim\limits_{m \to +\infty} \frac{s_m}{m} = + \infty,
    \end{equation}
    then it holds:
    \begin{equation}
        \label{controllability}
        \begin{aligned}
            \lim\limits_{t \to T_{-}} \|z(t)\|_{L^2(0, \be(t))^n} & = 0, \\
        \end{aligned}
    \end{equation}
\end{lemma}

\begin{proof}[Proof of Lemma~\ref{prop:finite}]

    From Proposition~\ref{prop:estimates}, it holds that for all $1\leq i \leq n$, for any $m \in \N$, for any $t \in [t_m,t_{m+1})$,
    \begin{equation*}
        \|k^{\sigma_i}_{\lambda_m}(t)\|_{L^2(D_t)} \leq C e^{c\be(t)\sqrt{\lambda_m/\sigma_i}},
    \end{equation*}
    \begin{equation*}
        \|l^{\sigma_i}_{\lambda_m}(t)\|_{L^2(D_t)} \leq C \left(\frac{\lambda_m}{\sigma_i}\right)^2 e^{c\be(t)}.
    \end{equation*}
    Fix $m \in \N$ and $t_m \leq t < t_{m+1}$. Denoting by $g_i(t):= \mathcal T_{\lambda_m,t}^{\sigma_i} z_i(t)$, it holds that
    \begin{align*}
        \|g_i(t)\|_{L^2(0, \be(t))}^2 & \leq \left(1+\|k^{\sigma_i}_{\lambda_m}(t)\|_{L^2(D_t)}^2\right) \|z_i(t)\|_{L^2(0,\be(t))}^2 \\
                                      & \leq C e^{c\be(t) \sqrt{\lambda_m/\sigma_i}} \|z_i(t)\|_{L^2}^2.
    \end{align*}
    Moreover, using the fact that $z_i(t) = \mathcal T_{\lambda_m,t}^{{\rm inv},\sigma_i} g_i(t)$, we obtain that
    \[ \|z_i(t)\|_ {L^2(0,\be(t))}^2 \leq \frac{C}{\sigma_i^4} \lambda_m^4 e^{c \be(t)} \|g_i(t)\|_{L^2(0,\be(t))}^2. \]
    Besides, from Proposition~\ref{prop:stability-target}, for any $t_m \leq \tau_1 \leq \tau_2 < t_{m+1}$, we have
    \[ \|g_i(\tau_2)\|_{L^2(0,\be(\tau_2))}^2 \leq e^{-2\lambda_m(\tau_2-\tau_1)} \|g_i(\tau_1)\|_{L^2(0,\be(\tau_1))}^2.\]
    Since $z_i$ and $g_i$ are in $C([0,T],L^2)$, we can combine these inequalities as $\tau_2 \to t_{m+1}^{-}$ and $\tau_1 \to t_m^+$. We thus obtain, with $C>0$ and $c>0$ being arbitrary constants independent of $t$,$i$ and $m$ which may change along the computations, and using the fact that $\ln(x) \leq \sqrt{x}$ for all $x>0$,
    \begin{align*}
        \|z_i(t_{m+1})\|_{L^2(0,\be(t_{m+1}))}^2 & \leq \frac{C}{\sigma_i^4} \lambda_m^4 e^{c \be(t_{m+1})} \|g_i(t_{m+1})\|_{L^2(0,\be(t_{m+1}))}^2                                            \\
                                                 & \leq Ce^{c \be(T) + 4 \ln(\lambda_m/\sigma_i)} \|g_i(t_{m+1})\|_{L^2(0,\be(t_{m+1}))}^2                                                      \\
                                                 & \leq Ce^{ 4 \log(\lambda_m/\sigma_i) - 2 \lambda_m (t_{m+1}-t_m)} \|g_i(t_m)\|_{L^2(0, \be(t_m))}^2                                          \\
                                                 & \leq Ce^{ 4 \log(\lambda_m/\sigma_i) + c \be(t_{m})\sqrt{\lambda_m/\sigma_i} - 2 \lambda_m (t_{m+1}-t_m)} \|z_i(t_m)\|_{L^2(0, \be(t_m))}^2, \\
                                                 & \leq Ce^{c\sqrt{\lambda_m/\sigma_i} - 2 \lambda_m (t_{m+1}-t_m)} \|z_i(t_m)\|_{L^2(0,\be(t_m))}^2.                                           \\
    \end{align*}
    Denoting by $\gamma:= \mathop{\max}_{1\leq i \leq n} \frac{c}{\sqrt{\sigma_i}}$, then, if \eqref{hyp1} holds, we obtain that for all $m\in \mathbb{N}$,
    \[ \|z(t_{m+1}) \|_{L^2(0, \be(t_{m+1})}^2 \leq Ce^{-\lambda_m (t_{m+1}-t_m)} \|z(t_m)\|_{L^2(0, \be(t_m))}^2 \leq Ce^{-s_{m} + \alpha m} \|z_0\|_{L^2(0, \be_0)}^2, \]
    with $\alpha:=\ln(C)$.
    This estimate together with \eqref{hyp2} yields \eqref{controllability} and the proof of the desired result.
\end{proof}

We are now in position to terminate the proof of Theorem~\ref{thm:2}. Indeed, for any $\gamma>0$, there always exist sequences $(t_m)_{m\in\N}$ and $(\lambda_m)_{m\in\N}$ that satisfy \eqref{hyp1} and \eqref{hyp2}.
Indeed, let us define, as in \cite{coron2017}, $t_m=T-\frac{1}{m^2}$ and $\lambda_m=\gamma^2 (m+1)^8$. Then, it holds that for all $m\geq 1$
\[ (t_{m+1}-t_m)\sqrt{\lambda_m}  = \gamma \frac{(2m +1)(m+1)^2 }{m^2}\geq \gamma. \]
Besides, $(t_{m+1}-t_m)\lambda_m = \gamma^2 \frac{(2m +1)(m+1)^6 }{m^2}$ so that $\displaystyle \frac{s_m}{m} \mathop{\longrightarrow}_{m\to +\infty} +\infty$.

Choosing such sequences, and defining the families of operators $(H_{nl}(t))_{t\geq 0}$ and $(H_l(t))_{t\geq 0}$ with \eqref{defHnl} and \eqref{defHl} then yields the desired result.

\section*{Conclusion and perspectives}
We have shown arbitrary small-time boundary stabilization for a class of cross-diffusion systems in a one-dimensional domain, at the level of the linearized system around uniform equilibria. The system is assumed to have an entropic structure and moreover its mobility matrix should be symmetric, so that the linearized system can be uncoupled into $n$ independent scalar equations. Anticipating on the nonlinear stabilization, we have chosen a weak $L^2$ framework for the stabilization. We have adapted the backstepping technique to derive a feedback control: we have shown that, although the equation is non autonomous, it suffices to study the usual stationary kernels PDEs in a moving domain, \emph{i.e.} the moving-domain structure is somehow transported to the kernel PDE. Besides, we have proven the well-posedness of the backstepping transformation in the framework of weak $L^2$ solutions and have provided quantitative estimates on the kernels with respect to time.

\medskip

We intend to continue this work to get the local stabilization of the nonlinear system. We also see several closely related open problems:

\medskip

\begin{itemize}
    \item  The symmetry assumption on the mobility matrix is technical. Without this assumption, one has to use the backstepping technique to stabilize the coupled linearized system~(\ref{linearized}). In consequence, one has to consider a matrix kernel $k$ with values in $\R^{n \times n}$ associated to the backstepping transformation. The derivation of the (matrix) kernel equations (see the scalar derivation in Appendix~\ref{app:derivation}) is complicated by the fact that $A(\bu)$ and $k$ \emph{do not commute} in general and leads to a ``non-commutative version'' of the kernel equations. On the other hand, the boundary conditions are unchanged since one obtains a commutation condition on the diagonal $x=y$. After the same separation of variables trick, one obtains the system:
    
\begin{equation}
    \label{eq:kpde_vec}
    \left\{
    \begin{aligned}
        \partial_{xx}^2 k(x,y) A(\bu) - A(\bu)\partial_{yy}^2 k(x,y) & = \lambda k(x,y) & (x,y) \in \{0 < y \leq x < +\infty\}, \\
        \partial_y k(x,0)   & = 0   & x \in (0,+\infty),  \\
        A(\bu) k(x,x)   & = -\frac{\lambda}{2} x    & x \in (0,+\infty),
    \end{aligned}
    \right.
\end{equation}

Up to our knowledge, it is an open problem to prove well-posedness and estimates for this system when $A(\bu)$ is not diagonalizable. 
\item The extension of the present work to the related nonlinear system is currently work in progress. We expect that getting global exponential or finite-time stabilization might be difficult in this situation. However, we have good hope of proving at least exponentially fast local stabilization. The control and estimates of the higher-order terms appearing in the equation is the most delicate part of the analysis.
\item It would be interesting, both mathematically and physically, to see whether it is possible to design an observer to have a control feedback that does not depend on the full state. An interesting additional direction would be to see whether the resulting observer-based control can be made robust (with respect to the propagation speeds of the system), which is not always granted (see for instance \cite{auriol2018delay,auriol2020robust,bastindiffusion}).
\item A last natural extension would be to study the stabilization of a similar system in a multidimensional context: this however requires as a first step to define a relevant multidimensional moving boundary domain model for the problem considered here. This is a very interesting problem left for future investigation. 
\end{itemize}

\section*{Acknowledgment}
The authors acknowledge support from the ANR project
COMODO (ANR-19-CE46-0002) which funds the Ph.D. of Jean Cauvin-Vila. The authors would also like to thank Shengquan Xiang for fruitful discussion. Finally, the authors would like to thank the PEPS JCJC program 2022 of INSMI.

\appendix
\appendixpage

\section{Weak formulation of the controlled linearized system in $L^2$}\label{app:weak}

We start from the strong formulation \eqref{linearized} with a feedback law of the form \eqref{eq:flux-eq}-\eqref{abstract-feedback}. We test against a regular test function $v$ that satisfies for all $x \in (0,T), ~ v(T,x)=0$ and integrate with respect to time and space. Considering the moving boundary, the integration of the time derivative gives:
\[ \int_0^T \int_0^{\be(t)} (\partial_t u \cdot v)(t,x) dx dt = -\int_0^T \int_0^{\be(t)} (\partial_t v \cdot u)(t,x)dxdt - \int_0^T \be(t)'(u \cdot v)(t,\be(t))dt - \int_0^{\be_0} u^0(x) \cdot v(0,x)dx. \]
Recall that $\be(t)'= \bv > 0$. Now we consider the space derivatives and perform two integration by parts:
\[ \begin{aligned}
        \int_0^T \int_0^{\be(t)} A(\bu)\partial_{xx}^2u \cdot v dxdt = & \int_0^T A(\bu) \partial_x u(t,\be(t)) \cdot v(t,\be(t)) dt - \int_0^T \int_0^{\be(t)}A(\bu)\partial_x u \cdot \partial_x v dxdt \\
        =                                                              & \int_0^T A(\bu) \partial_x u(t,\be(t)) \cdot v(t,\be(t)) dt + \int_0^T \int_0^{\be(t)} A(\bu) u \cdot \partial_{xx}^2 v dxdt     \\
                                                                       & - \int_0^T A(\bu) u(t,\be(t) \cdot \partial_x v(t,\be(t)) dt + \int_0^T A(\bu)u(t,0) \cdot \partial_x v(t,0) dt.
    \end{aligned}   \]
Now we write the equality of these two quantities with the appropriate factorizations:
\[ \begin{aligned}
        0 = & \int_0^T \int_0^{\be(t)} u \cdot \left[ \partial_t v + A(\bu)^T \partial_{xx}v \right] dxdt + \int_0^{\be_0} u^0(x) \cdot v(0,x)dx - \int_0^T  A(\bu) u(t,\be(t)) \cdot \partial_x v(t,\be(t))dt \\
            & + \int_0^T A(\bu) u(t,0) \cdot  \partial_x v(t,0) dt + \int_0^T v(t,\be(t)) \cdot \left[ A(\bu)\partial_x u(t,\be(t)) + \bv u(t,\be(t)) \right] dt .
    \end{aligned} \]
In the last integral we recognize the boundary condition at $x=\be(t)$, that is nothing else than $\delta \psi(t) = H_{nl}(t) u(t) + K_l(t)u(t,\be(t))$. Now all the terms that do not make sense for $u \in \left[\cC^0([0,+\infty),L^{2}(0,1))^n\right]_{\overline{e}}$ must vanish if this is to be true against any test function. It entails conditions on the test functions, the so-called dual boundary conditions. The first condition at $x=0$ is:
\[ \forall t \in (0,T), ~ A(\bu)^T \partial_x v(t,0) = 0. \]
Now we examine the condition at $x=\be(t)$ where the local part of the feedback intervenes:
\[ \forall t \in (0,T), ~ K(t)^T v(t,\be(t)) - A(\bu)^T \partial_x v(t,\be(t)) = 0. \]
The remaining terms make sense for $u \in \left[\cC^0([0,+\infty),L^{2})^n\right]_{\overline{e}}$ provided:
\[ v \in \left[L^2\left((0,T); H^2\right)^n\right]_{\overline{e}}\cap \left[\cC^0([0,T],L^{2})^n\right]_{\overline{e}}, \; \partial_t v\in \left[L^2\left((0,T); L^2\right)^n\right]_{\overline{e}}. \]

Putting together all the conditions, we obtain Definition~\ref{def:weak-solution}.

\section{Analysis of the target problem}\label{app:target}

For the sake of the analysis, we consider the rescaled version of \eqref{target} given by the change of variables $x\rightarrow x/\bar e(t)$ so that the space variable is now defined in a fixed domain:
\begin{equation}
    \label{rescaled-target}
    \left\{
    \begin{aligned}
        \partial_t w - \frac{\sigma}{\be(t)^2} \partial^2_{xx} w - \frac{\bv}{\be(t)} x \partial_x w + \lambda w & = 0,                      & \text{for} ~ (t,x) \in \R_+^* \times (0,1), \\
        \frac{\sigma}{\be(t)} \partial_x w(t,1) + \bv w(t,1)                                                     & = 0,                      & \text{for} ~ t \in \R_+^*,                  \\
        \frac{\sigma}{\be(t)} \partial_x w(t,0)                                                                  & = 0,                      & \text{for} ~ t \in \R_+^*,                  \\
        w(0,x)                                                                                                   & = w^0(x) := g^0(x \be_0), & \text{for} ~ x \in (0,1),
    \end{aligned}
    \right.
\end{equation}
and the associated notion of weak $L^2$ solution:
\begin{definition}\label{def:rescaled-weak-target}
    A function $w \in \cC^0([0,+\infty),L^{2}(0,1))$ is said to be a $L^2$-weak solution of \eqref{rescaled-target} if for any $T > 0$, it satisfies:

    \[
        \int_{0}^{T}\int_{0}^{1} w(t,x) \left[\partial_{t}\tilde{v}(t,x)+ \frac{\sigma}{\be(t)^2} \partial_{xx}^{2}\tilde{v}(t,x) - \frac{\bv}{\be(t)}x \partial_x \tilde{v}(t,x) - \left(\lambda+\frac{\bv}{\be(t)} \right) \tilde{v}(t,x) \right] dxdt  + \int_{0}^{1}w(0,x) \tilde{v}(0,x) dx  = 0,
    \]

    for any test function $\tilde{v}$ that satisfies:
    \begin{itemize}
        \item $\tilde{v} \in \left(L^2\left((0,T); H^2(0,1)\right)\right)\cap \cC^0([0,T],L^{2}(0,1°))$,
        \item $\partial_t \tilde{v}\in \left(L^2\left((0,T); L^2\right)\right)$,
        \item $\tilde{v}(T,\cdot) = 0$,
        \item $\sigma \partial_{x}\tilde{v}(t,0) = 0, ~ \forall t \in (0,T)$,
        \item $\sigma \partial_{x}\tilde{v}(t,\be(t)) = 0, \forall t \in (0,T)$.
    \end{itemize}

\end{definition}
Note that the two definitions \ref{def:weak-target} and \ref{def:rescaled-weak-target} are equivalent: from the latter to the former, take a test function of the form $\be(t) v(t,\be(t)x)$ where $v$ satisfies the assumptions in \ref{def:weak-target}. The other way around, take a test function of the form $\frac{1}{\be(t)} \tilde{v}(t,\frac{x}{\be(t)})$ where $\tilde{v}$ satisfies the previous assumptions.

\noindent The problem is uniformly parabolic: for any $0 \leq t \leq T$
\[ \frac{\sigma}{\be(t)^2} \geq \frac{\sigma}{\be(T)^2} > 0,  \]
so that we expect the classical parabolic estimates and well-posedness in $\cC^0([0,T],H^{k}(0,1))$ for any $k \in \N$ as soon as $w^0 \in H^k(0,1)$. In fact, we have the following a priori estimates

\begin{lemma}
    \label{lem:smooth-estimation}
    Assume $w^0 \in L^2((0,1))$. Any smooth solution $w$ to \eqref{rescaled-target} must satisfy the energy estimate:
    \begin{equation}
        \begin{split}
            \label{eq:energy-estimate}
            &\frac{1}{2}\|w\|_{L^{\infty}(0,T;L^2)}^2 + \frac{\sigma}{\be(T)^2} \|\partial_x w \|_{L^2(0,T;L^2(0,1))}^2 \\
            +& \frac{\bv}{2\be(T)} \int_0^T w(t,1)^2 + \left(\frac{\bv}{2\be(T)} + \lambda \right) \|w\|_{L^2(0,T;L^2(0,1))}^2 \leq \|w^0\|_{L^2(0,1)}^2.
        \end{split}
    \end{equation}
    Furthermore, it must satisfy the stability estimate: for any $0 \leq t \leq T$
    \begin{equation}
        \label{eq:stability-estimate}
        \| w(t) \|_{L^2(0,1)} \leq e^{-\lambda t} \|w^0\|_{L^2(0,1)}
    \end{equation}
\end{lemma}

\begin{proof}
    Multiply the first equation in \eqref{rescaled-target} by $w$ and integrate by parts in space at time $t$. One first obtains:
    \begin{equation}
    \label{eq:first}
        \frac{1}{2} \frac{d}{dt} \|w(t)\|_{L^2}^2 + \frac{\sigma}{\be(t)^2} \int_0^1 (\partial_x w)^2 + \left(\frac{\bv}{2\be(t)} + \lambda \right) \int_0^1 w^2  - \frac{\sigma}{\be(t)^2} \partial_x w(t,1) w(t,1) - \frac{\bv}{2\be(t)} w(t,1)^2 = 0.
    \end{equation}
    Then using the boundary conditions in \eqref{rescaled-target}, one gets:
    \begin{equation}
    \label{eq:apriori}
    \begin{split}
         \frac{1}{2} \frac{d}{dt} \|w(t)\|_{L^2}^2 + \frac{\sigma}{\be(t)^2} \int_0^1 (\partial_x w)^2 + \left(\frac{\bv}{2\be(t)} + \lambda \right) \int_0^1 w^2  + \frac{\bv}{2\be(t)} w(t,1)^2 = 0.
    \end{split}
    \end{equation}
    It comes in particular:
    \[ \frac{1}{2}\frac{d}{dt} \|w(t)\|_{L^2(0,1)}^2 \leq -\lambda \|w(t)\|_{L^2(0,1)}^2, \]
    from which we conclude to \eqref{eq:stability-estimate} with the Gronwall lemma. Integrating \eqref{eq:apriori} with respect to time in $[0,T]$, one finds \eqref{eq:energy-estimate}.

\end{proof}

From these estimates, we define a notion of \emph{energy solution} for the problem.

\begin{definition}
    \label{def:energy-solution}
    A function $w \in \cC^0([0,T],L^2(0,1)) \cap L^2((0,T);H^1(0,1))$ such that $\partial_t w \in L^2((0,T);(H^1(0,1))')$ is an energy solution to \eqref{rescaled-target} if, for almost any time $0 \leq t \leq T$ and any function $\tilde{v} \in H^1(0,1)$, it satisfies
    \begin{equation}
        \label{eq:energy-ff}
        \langle \partial_t w(t), \tilde{v}(t) \rangle_{(H^1)',H^1} + a(t;w,\tilde{v}) = 0,
    \end{equation}
    where the bilinear form $a$ is given by:
    \begin{equation}
        \label{eq:bilinear}
        a(t;w,\tilde{v}) = \int_0^1 \left( \frac{\sigma}{\be(t)^2} \partial_x w \partial_x \tilde{v} - \frac{\bv}{\be(t)} w \left(\tilde{v}+x \partial_x \tilde{v}\right) + \lambda w \tilde{v} \right)dx.
    \end{equation}
\end{definition}

By construction of the weak formulation, such solutions still satisfy the previous estimates. In particular and by linearity, such a solution is unique, if it exists. The existence follows from the Galerkin method (see \cite{brezis2010}[Theorem 10.9, p.341] for a general result). This is summarized in the following proposition

\begin{proposition}
    \label{prop:existence}
    Let $w^0 \in L^2(0,1)$. There exists a unique energy solution to \eqref{rescaled-target}. This solution satisfies \eqref{eq:energy-estimate} and \eqref{eq:stability-estimate}.
\end{proposition}

Proposition~\ref{prop:existence} gives existence to a weak $L^2$ solution since the energy solution is a particular one, and this solution satisfies in particular \eqref{eq:stability-estimate}. But one cannot directly conclude that \emph{any} weak $L^2$ satisfies \eqref{eq:stability-estimate} and deduce uniqueness from the estimate. Indeed, such an estimate cannot be deduced directly from the weak formulation in \eqref{def:weak-target}. Instead, we will first prove uniqueness from the weak formulation, then deduce that the only weak $L^2$ solution satisfies indeed the estimate.

\begin{lemma}
    \label{lem:uniqueness}
    There is at most one weak $L^2$ solution in the sense of Definition~\ref{def:rescaled-weak-target}.
\end{lemma}

\begin{proof}
    Consider two such solutions $w_1,w_2 \in \cC^0([0,T];L^2(0,1))$. Then the difference satisfies, for any test function $\tilde{v}$ that satisfies the assumptions of Definition~\ref{def:rescaled-weak-target}:
    \[ \int_0^T \int_0^1 (w_1-w_2)\left(\partial_t \tilde{v} + \frac{\sigma}{\be(t)^2} \tilde{v} - \frac{\bv}{\be(t)} x \partial_x \tilde{v} - \left(\lambda + \frac{\bv}{\be(t)}\right) \tilde{v} \right) dx dt = 0 \]
    Now fix $S \in L^2((0,T);L^2(0,1))$ and consider the inhomogeneous dual problem with source term $S$:
    \begin{equation}
        \label{rescaled-dual}
        \left\{
        \begin{aligned}
            \partial_t \tilde{v} + \frac{\sigma}{\be(t)^2} \partial^2_{xx} \tilde{v} - \frac{\bv}{\be(t)} x \partial_x \tilde{v} - \left(\lambda + \frac{\bv}{\be(t)}\right) \tilde{v} & = S, & \text{for} ~ (t,x) \in [0,T] \times (0,1), \\
            \partial_x \tilde{v}(t,1)                                                                                                                                                  & = 0, & \text{for} ~ t \in [0,T],                  \\
            \partial_x \tilde{v}(t,0)                                                                                                                                                  & = 0, & \text{for} ~ t \in [0,T],                  \\
            \tilde{v}(T,x)                                                                                                                                                             & = 0, & \text{for} ~ x \in (0,1).
        \end{aligned}
        \right.
    \end{equation}
    Up to time reversal $t \to T-t$, this is a classical parabolic problem with smooth coefficients. Therefore, there exists a (unique) solution $\tilde{v} \in \left(L^2\left((0,T); H^2(0,1)\right)\right)\cap \cC^0([0,T],L^{2}(0,1))$ and such that $\partial_t \tilde{v}\in \left(L^2\left((0,T); L^2(0,1)\right)\right)$ (the regularity is limited by $S \in L^2$). Consequently, $\tilde{v}$ can be taken as a test function against $(w_1-w_2)$ and it holds:
    \[ \int_0^T \int_0^1 (w_1-w_2)S dx dt = 0. \]
    Since this is true for any $S \in L^2((0,T);L^2(0,1))$, $w_{1}=w_{2}$ and uniqueness is proved.

\end{proof}

Proposition~\ref{prop:stability-target} follows from Proposition~\ref{prop:existence} and \ref{lem:uniqueness}. 

\medskip

\section{Formal derivation of the backstepping kernel problems}\label{app:derivation}
As explained in the remarks in Section~\ref{sec:kernels}, the derivation is done assuming that the kernels explicitly depend on time $t$, then we show they do not have to depend on $t$. In the spirit of Section~\ref{sec:feedback}, we assume that all the functions are smooth and we differentiate \eqref{eq:map} at $x=0$. It gives:
\[ \partial_x g(t,0) = -k(t,0,0) \zeta(t,0), \]
which suggests, since $\zeta(t,0)$ is undetermined, that the kernel $k$ should be supplied with the condition
\begin{equation}
    \label{eq:corner}
    k(t,0,0)=0,
\end{equation}
for the boundary condition at $x=0$ in \eqref{target} to be satisfied.

At this stage, it is unclear how the kernel depends on time. We make it explicit by applying a rescaling in the space variable into a fixed domain. More precisely, we consider the rescaled versions of problems \eqref{linearized_scal} and \eqref{target}. The latter one was defined in \eqref{rescaled-target} while the former is given by:
\begin{equation}
    \label{rescaled-linearized}
    \left\{
    \begin{aligned}
        \partial_t z - \frac{1}{\be(t)^2} \sigma \partial^2_{xx} z - \frac{\bv}{\be(t)} x \partial_x z & = 0,                 & \text{for} ~ (t,x) \in \R_+^* \times (0,1), \\
        \frac{\sigma}{\be(t)}  \partial_x z(t,1) + \bv z(t,1)                                          & = \delta \psi(t),    & \text{for} ~ t \in \R_+^*,                  \\
        \frac{\sigma}{\be(t)} \partial_x z(t,0)                                                        & = 0,                 & \text{for} ~ t \in \R_+^*,                  \\
        z(0,x)                                                                                         & = \zeta^0(x \be(t)), & \text{for} ~ x \in (0,1).
    \end{aligned}
    \right.
\end{equation}

We consider the backstepping transformation associated to these rescaled problems: for any $t \geq 0$, $x \in (0,1)$:
\begin{equation}
    \label{rescaled-map}
    w(t,x) := z(t,x) - \int_0^x \tk(t,x,y) z(t,y) dy,
\end{equation}
where $z$ is a solution to \eqref{rescaled-linearized}, $w$ a solution to \eqref{rescaled-target} and $\tk$ an unknown function defined in $D_1 := \{0 < y \leq x < 1 \}$. Assume everything is smooth and compute derivatives:
\begin{align}
    \label{gx}
    \partial_x w(t,x)      & = \partial_x z(t,x) - \tk(t,x,x) z(t,x) - \int_0^x  \partial_x \tk(t,x,y)z(t,y)dy,                                                                \\
    \partial^2_{xx} w(t,x) & =  \partial_{xx}^2 z(t,x) - \partial_x\left(\tk(t,x,x)z(t,x)\right) - \partial_x \tk(t,x,x)z(t,x) - \int_0^x  \partial^2_{xx} \tk(t,x,y)z(t,y)dy.
\end{align}
and:
\begin{align*}
    \partial_t w(t,x) & = \partial_t z - \int_0^x \tk(t,x,y)\partial_t z(t,y)dy - \int_0^x \partial_t \tk(t,x,y)z(t,y)dy                                                                                                                                           \\
                      & = \frac{\sigma}{\be(t)^2} \partial_{xx}^2 z + \frac{\bv}{\be(t)} x\partial_x z - \int_0^x \tk(t,x,y)\left(\frac{\sigma}{\be(t)^2} \partial_{yy}^2 z + \frac{\bv}{\be(t)} y \partial_y z \right)dy - \int_0^x \partial_t \tk(t,x,y)z(t,y)dy
\end{align*}
Now use integration by parts for the integral terms in the middle. It holds:
\begin{equation*}
    \int_0^x \tk(t,x,y)y \partial_x z(t,y) dy = \tk(t,x,x)xz(t,x) - \int_0^x \partial_y (y\tk(t,x,y)) z(t,y)dy,
\end{equation*}
and, using $\partial_x z(0)=0$:
\begin{align*}
    -\int_0^x \tk(t,x,y)\partial_{yy}^2 z & = -\tk(t,x,x)\partial_x z(t,x) + \int_0^x \partial_y \tk(t,x,y) \partial_x z dy                                                                       \\
                                          & = -\tk(t,x,x)\partial_x z(t,x) + \left(\partial_y \tk(t,x,x)z(t,x)-\partial_y \tk(t,x,0)z(t,0) \right) - \int_0^x \partial_{yy}^2 \tk(t,x,y)z(t,y)dy,
\end{align*}
so that it holds:
\begin{equation}
    \label{gt}
    \begin{split}
        \partial_t w(t,x) =& \frac{\sigma}{\be(t)^2} \left(\partial_{xx}^2 z - \tk(t,x,x)\partial_x z(t,x) + \left(\partial_y \tk(t,x,x)z(t,x)-\partial_y \tk(t,x,0)z(t,0) \right)\right.\\
        &-\left. \int_0^x \partial_{yy}^2 \tk(t,x,y)z(t,y)dy \right) \\ &+\frac{\bv}{\be(t)}\left(x\partial_x z + \int_0^x \partial_y (y \tk(t,x,y)) z(t,y)dy -\tk(t,x,x)xz(x)\right) - \int_0^x \partial_t \tk(t,x,y)z(t,y)dy
    \end{split}
\end{equation}
Now we insert \eqref{gx},\eqref{gt} into the equation satisfied by $w$:
\[\partial_t w - \frac{\sigma}{\be(t)^2} \partial_{xx}^2 w - \frac{\bv}{\be(t)} x  \partial_x w + \lambda w = 0. \]
After cancellations, it remains, for any $x \in (0,1)$:
\begin{equation*}
    \begin{split}
        0 = \frac{\sigma}{\be(t)^2} \left(2\frac{d}{dx} \tk(t,x,x) + \lambda \right) z(x) - \partial_y \tk(t,x,0)z(0) \\
        - \int_0^x z\left(\partial_t \tk-\frac{\sigma}{\be(t)^2}(\partial^2_{xx} \tk-\partial_{yy}^2 \tk) - \frac{\bv}{\be(t)} (x\partial_x \tk+\partial_y (y\tk)) + \lambda \tk \right)dy,
    \end{split}
\end{equation*}
which leads to the following problem for $t>0$
\begin{equation}
    \label{rescaled-kernel_eq}
    \left\{
    \begin{aligned}
        \partial_t \tk - \frac{\sigma}{\be(t)^2} \left( \partial_{yy}^2 \tk(t,x,y) - \partial_{xx}^2 \tk(t,x,y) \right) - \frac{\bv}{\be(t)} \left(x \partial_x \tk + y \partial_y \tk + \tk \right) + \lambda \tk & = 0         & (x,y) \in \{0 < y \leq x < 1\}, \\
        \frac{\sigma}{\be(t)^2} \partial_y \tk(t,x,0)                                                                                                                                                              & = 0         & x \in (0,1),                    \\
        \frac{2 \sigma}{\be(t)^2} \frac{d}{dx} \tk(t,x,x)                                                                                                                                                          & = - \lambda & x \in (0,1).
    \end{aligned}
    \right.
\end{equation}
Now we look for a solution with separate variables under the form ($k$ does not depend explicitly on time):
\begin{equation}
    \label{separate-variables}
    \tk(t,x,y) = \be(t) k(x\be(t),y\be(t)).
\end{equation}
Inserting \eqref{separate-variables} into \eqref{rescaled-kernel_eq}, the terms in $\be(t)$ cancel each other and coming back to the original domain we obtain 
\begin{equation}
    \label{eq:kernel_eq-Dt}
    \left\{
    \begin{aligned}
        \partial_{xx}^2 k(x,y) - \partial_{yy}^2 k(x,y) & = \frac{\lambda}{\sigma} k(x,y) & (x,y) \in \{0 < y \leq x < \be(t)\}, \\
        \partial_y k(x,0)                                                & = 0                                             & x \in (0,\be(t)),  \\
        k(x,x)                                                           & = -\frac{\lambda}{2 \sigma} x                   & x \in (0,\be(t)),
    \end{aligned}
    \right.
\end{equation} 
Moreover, inserting \eqref{separate-variables} into \eqref{rescaled-map}, it is clear with a change of variables that the $k$ defined from \eqref{separate-variables} enables to recover the original kernel in \eqref{eq:map}-\eqref{eq:defT} we were looking for.

\section{Failure of the basic quadratic Lyapunov approach}
\label{sec:lyapunov}
We show why the common approach of directly using a basic quadratic Lyapunov function would fail to provide exponential stabilization. In order to have a proper basic quadratic Lyapunov function, we work on the rescaled system \eqref{rescaled-linearized}. A basic quadratic Lyapunov function for the $L^{2}$ norm has the form, for some positive function $f \in C^2((0,1)) \cap C^1([0,1])$,
\begin{equation}
    \label{eq:quadratic-Lyapunov}
    V(z(t,\cdot)) = \int_{0}^{1} f(x)z(t,x)^2dx.
\end{equation}

Let us take $V$ as a Lyapunov function candidate. By differentiating along $C^{2}$ solutions of \eqref{rescaled-linearized}, we have
\begin{equation}
    \frac{d}{dt}V(z(t,\cdot)) = \int_{0}^{1} 2f(x)z\left[ \frac{ \sigma}{\be(t)^2} \partial_{xx}^2 z + \frac{\bv}{\be(t)} x \partial_x z \right] dx,
\end{equation}
Integrating by parts the first term gives
\[ \left[2f\frac{ \sigma}{\be(t)^2}z\partial_{x}z\right]_{0}^{1} - \frac{ \sigma}{\be(t)^2} \int_{0}^{1} \left[2f(x)(\partial_{x}z)^{2}+2f'(x)z\partial_{x}z \right] dx, \]
while the second terms gives
\[ \left[fx\frac{\bv}{\be(t)}z^{2}\right]_{0}^{1} - \frac{2\bv}{\be(t)} \int_0^1 \left(z^2 \left(f + xf' \right) + f x z \partial_x z \right)dx.   \]
Note that the last term in the previous equation is the same term we integrated by parts. Therefore putting everything together we obtain:
\begin{equation}
    \begin{split}
        \frac{d}{dt}V(z(t,\cdot)) = \left(\left[2f\frac{ \sigma}{\be(t)^2}z\partial_{x}z\right]_{0}^{1}+\left[fx\frac{\bv}{\be(t)}z^{2}\right]_{0}^{1}\right)- \frac{ \sigma}{\be(t)^2} \int_{0}^{1}& \left[2f(x)(\partial_{x}z)^{2}+2f'(x)z\partial_{x}z \right] dx\\
        &- \frac{\bv}{\be(t)} \int_0^1 z^2 \left(f + xf' \right)dx,
    \end{split}
\end{equation}
which gives, using again an integration by parts for the second term in the first integral and the boundary conditions of \eqref{rescaled-linearized}
\begin{equation}
    \begin{split}
        \frac{d}{dt}V(z(t,\cdot)) =& \left(\left[2f\frac{ \sigma}{\be(t)^2}z\partial_{x}z\right]_{0}^{1}+\left[fx\frac{\bv}{\be(t)}z^{2}\right]_{0}^{1}-\left[z^{2}f'\frac{ \sigma}{\be(t)^2}\right]_{0}^{1}\right)\\
        &- \frac{\sigma}{\be(t)^{2}} \int_{0}^{1}\left[ 2f(x)(\partial_{x}z)^{2}-f''(x)z^{2} \right] dx
        -\frac{\bv}{\be(t)} \int_0^1   z^2 \left(f + xf' \right)dx.\\
        =& \left(\frac{2}{\be(t)}f(1)\delta \psi(t)z(t,1)-\frac{\bv}{\be(t)}z^{2}(t,1))+\frac{ \sigma}{\be(t)^2} \left(z(t,0)^2f'(0) -z(t,1)^2f'(1) \right) \right)\\
        &- \frac{\sigma}{\be(t)^{2}} \int_{0}^{1}\left[ 2f(x)(\partial_{x}z)^{2}-f''(x)z^{2} \right] dx
        -\frac{\bv}{\be(t)} \int_0^1 z^2 \left(f + xf' \right)dx.
    \end{split}
\end{equation}
To have a Lyapunov function ensuring an exponential stability estimate, there has to exist $\gamma>0$ such that the right-hand side is lower or equal than $-\gamma V$ for any $t\in[0,T]$ and any solution of \eqref{rescaled-linearized}.
From that point one would typically require in the Lyapunov approach that for all $t\in[0,+\infty)$ and $\;Z\in C^{2}([0,1])$,
\begin{equation}
    \label{eq:necessaire-cond}
    \begin{split}
        & \left(\frac{2}{\be(t)}f(1) \delta \psi(t)Z(1)-\frac{\bv}{\be(t)}Z^{2}(1)) + \frac{ \sigma}{\be(t)^2} \left(Z(0)^2f'(0)-Z(1)^2f'(1) \right) \right)\\
        &- \int_{0}^{1}\left[2f(x)\frac{\sigma}{\be(t)^{2}}(\partial_{x}Z)^{2}
            +\left(\frac{\bv}{\be(t)} (f+ xf')-f''(x)\frac{ \sigma}{\be(t)^2}-\gamma f(x)\right) Z^{2}\right]dx\leq 0,
    \end{split}
\end{equation}
In particular this would be true for any $Z\in C^{2}([0,1])$ with compact support which implies that
\begin{equation}
    \begin{split}
        \int_{0}^{1}\left[2f(x)\frac{\sigma}{\be(t)^{2}}(\partial_{x}Z)^{2}
            +\left(\frac{\bv}{\be(t)} (f+ xf')-f''(x)\frac{ \sigma}{\be(t)^2}-\gamma f(x)\right) Z^{2}\right]dx \geq 0.
    \end{split}
\end{equation}
Since this has to be true for any time and any $Z\in C_{c}^{2}([0,1])$, and since $\be(t)\rightarrow +\infty$ when $t\rightarrow+\infty$, this implies that for any $x\in (0,1)$,
\begin{equation}
    \begin{split}
        xf'(x)-\gamma f(x)\geq0,
    \end{split}
\end{equation}
but as $f\in C^{1}([0,1];(0,+\infty))$ this is impossible: indeed, denoting $M =\sup_{[0,1]}(f')\in\mathbb{R}$ and $m=\inf_{[0,1]}(f)>0$ this would imply in particular that
\begin{equation}
    xM\geq \gamma m>0,\;\forall x\in(0,1),
\end{equation}
which would lead to a contradiction. Note that, although Lyapunov functionals of the form \eqref{eq:quadratic-Lyapunov} fail here, some other Lyapunov functionals (\cite{coron2004global,xiang2020quantitative} or quadratic functionals with time-dependent weights) may manage to provide rapid stabilization results for this system.

%Bibliography
\bibliographystyle{plain}
\bibliography{full}

\end{document}